\documentclass[11pt,a4paper]{amsart}

\usepackage[english]{babel}
\usepackage[utf8]{inputenc}
\usepackage{amsmath}
\usepackage{amssymb}
\usepackage{amsthm}
\usepackage{bm}
\usepackage{dsfont}
\usepackage{enumitem}
\usepackage{scalerel}
\usepackage{mathtools}
\usepackage{fullpage}
\usepackage{xcolor}

\theoremstyle{plain}
\newtheorem{theorem}{Theorem}[section]
\newtheorem{corollary}[theorem]{Corollary}
\newtheorem{lemma}[theorem]{Lemma}
\newtheorem{proposition}[theorem]{Proposition}

\newtheorem{question}{Question}[section]

\theoremstyle{definition}

\newtheorem{definition}[theorem]{Definition}
\newtheorem{notation}[theorem]{Notation}

\theoremstyle{remark}
\newtheorem{remark}[theorem]{Remark}

\newcommand{\N}{\mathbb{N}}

\newcommand{\Q}{\mathbb{Q}}

\newcommand{\C}{\mathbb{C}}

\newcommand{\A}{\mathbb{A}}
\newcommand{\B}{\mathbb{B}}

\newcommand{\G}{\mathbb{G}}

\newcommand{\K}{\mathbb{K}}

\newcommand{\TT}{\mathcal{T}}

\newcommand{\set}[2]{\left\{#1:#2\right\}}

\newcommand{\cupdot}{\mathbin{\dot{\cup}}}

\newcommand{\abs}[2][]{\left|#2\right|_{#1}}

\newcommand{\actson}{\curvearrowright}

\DeclareMathOperator{\sgn}{sgn}

\DeclareMathOperator{\id}{id}

\DeclareMathOperator{\Img}{Im}
\DeclareMathOperator{\Dom}{Dom}
\DeclareMathOperator{\Sym}{Sym}
\DeclareMathOperator{\Aut}{Aut}
\DeclareMathOperator{\Inj}{Inj}
\DeclareMathOperator{\Emb}{Emb}
\DeclareMathOperator{\End}{End}

\DeclareMathOperator{\Stab}{Stab}
\DeclareMathOperator{\Orb}{Orb}

\usepackage{hyperref}
\hypersetup{colorlinks=true,linkcolor=blue,citecolor=red}
\usepackage{tikz}
\usetikzlibrary{matrix,arrows,calc,babel,patterns,decorations.pathmorphing,decorations.pathreplacing}

\theoremstyle{plain}

\title{On the Zariski topology on endomorphism monoids of
  omega\nobreakdash-categorical structures}
\author{Michael Pinsker}
\thanks{This research was funded in whole or in part by the Austrian
  Science Fund (FWF) [P 32337, I 5948]. For the purpose of Open
  Access, the authors have applied a CC BY public copyright licence to
  any Author Accepted Manuscript (AAM) version arising from this
  submission. This research is also funded by the European Union (ERC,
  POCOCOP, 101071674). Views and opinions expressed are however those
  of the author(s) only and do not necessarily reflect those of the
  European Union or the European Research Council Executive Agency.
  Neither the European Union nor the granting authority can be held
  responsible for them.}
\author{Clemens Schindler}
\thanks{The second author is a recipient of a DOC Fellowship of the
  Austrian Academy of Sciences at the Institute of Discrete
  Mathematics and Geometry, TU Wien.}

\address{Institut für Diskrete Mathematik und Geometrie, FG Algebra,
  TU Wien, Austria}
\email{marula@gmx.at}
\address{Institut für Diskrete Mathematik und Geometrie, FG Algebra,
  TU Wien, Austria}
\email{clemens.schindler@tuwien.ac.at}

\keywords{Reconstruction, Zariski topology, endomorphism monoid,
  pointwise convergence topology, model-complete core}

\begin{document}
\begin{abstract}
  The endomorphism monoid of a model-theoretic structure carries two
  interesting topologies: on the one hand, the topology of pointwise
  convergence induced externally by the action of the endomorphisms on
  the domain via evaluation; on the other hand, the Zariski topology
  induced within the monoid by (non\nobreakdash-)solutions to
  equations. For all concrete endomorphism monoids of
  $\omega$-categorical structures on which the Zariski topology has
  been analysed thus far, the two topologies were shown to coincide,
  in turn yielding that the pointwise topology is the coarsest
  Hausdorff semigroup topology on those endomorphism monoids.

  We establish two systematic reasons for the two topologies to agree,
  formulated in terms of the model-complete core of the
  structure. Further, we give an example of an $\omega$-categorical
  structure on whose endomorphism monoid the topology of pointwise
  convergence and the Zariski topology differ, answering a question of
  Elliott, Jonu\v{s}as, Mitchell, Péresse and Pinsker.
\end{abstract}

\maketitle
\section{Introduction}
\label{sec:introduction}

\subsection{Motivation}
\label{sec:motivation}
Given a model-theoretic (relational) structure $\A$ with domain $A$,
the set $\End(\A)$ of all endomorphisms of $\A$ is closed under
composition of functions and thus forms a semigroup (even a
monoid). Inheriting the subspace topology of the product topology on
$A^{A}$ where each copy of $A$ is equipped with the discrete topology,
$\End(\A)$ additionally carries a topological structure which turns
out to be Polish, i.e.~separable and completely metrisable, in
particular Hausdorff. In this topology, a sequence $(f_{n})_{n\in\N}$
converges to $f$ if and only if for every $a\in A$, the sequence of
evaluations $(f_{n}(a))_{n\in\N}$ converges to $f(a)$ in the discrete
topology, i.e. if it is eventually constant with value $f(a)$. For
this reason, the topology is called the \emph{topology of pointwise
  convergence} or \emph{pointwise topology} for brevity. These two
types of structure are compatible in the sense that the composition
operation is continuous with respect to the pointwise topology; one
says that the topology is a \emph{semigroup topology}. For many
model-theoretic structures $\A$ on a countable domain, the algebraic
(semigroup) structure and the topological (Polish) structure turn out
to be so deeply intertwined that the pointwise topology is the
\emph{unique} Polish semigroup topology on $\End(\A)$. Examples
include the structure without relations (whose endomorphism monoid is
the full transformation monoid)~\cite{EJMMMP-zariski}; the random
(di\nobreakdash-)graph, the random strict partial order and the
equivalence relation with either finitely or countably many
equivalence classes of countably infinite size~\cite{EJMPP-polish}; as
well as the rational numbers with the non-strict order~\cite{PS-EndQ}.

One obvious step in the proofs of these results is to show that the
pointwise topology is the coarsest Polish semigroup topology on
$\End(\A)$. For this purpose, the authors of~\cite{EJMMMP-zariski}
transferred a notion from the theory of topological groups to the
realm of semigroups, namely the so-called \emph{Zariski topology} (or
sometimes \emph{verbal topology}),
see~\cite{Bryant-verbal,DikranjanToller-zariskimarkov,Markov-three-papers};
roughly speaking, the closed sets in this topology are given by
solution sets to identities in the language of semigroups. Hence, the
Zariski topology is an object associated to the algebraic
(semi\nobreakdash-)group structure. Considering $\End(\A)$ as an
abstract semigroup, the Zariski topology can thus be regarded as an
``internal'' object. The pointwise topology, in contrast, is defined
from the evaluations at elements of the domain of $\A$ and is thus an
``external'' object with respect to the abstract semigroup structure
of $\End(\A)$ -- precisely speaking, the pointwise topology is
associated to the semigroup action of $\End(\A)$ on $A$.

As it turns out, the Zariski topology is necessarily coarser than any
Hausdorff semigroup topology on a given semigroup. In particular, the
pointwise topology on $\End(\A)$ is always finer than the Zariski
topology. If one manages to show that the Zariski topology on
$\End(\A)$ even coincides with the pointwise topology for some
structure $\A$, one can draw two conclusions: on the one hand, the
pointwise topology can also be understood as an ``internal'' object
with respect to the abstract semigroup structure; on the other hand,
the pointwise topology then indeed is the coarsest (in particular)
Polish semigroup topology on $\End(\A)$. This technique was used
in~\cite{EJMMMP-zariski} and~\cite{EJMPP-polish} as well as,
implicitly, in~\cite{PS-EndQ}. In each instance, however, the proof
that the topologies coincide has not been particularly systematic but
tuned to the specific situation being considered, based on two sets of
rather technical sufficient conditions established
in~\cite{EJMMMP-zariski} and the ad hoc notion of so-called
\emph{arsfacere} structures introduced in~\cite{EJMPP-polish} for
which these conditions always hold. This raises the question whether
there are systematic reasons for equality of the topologies, in other
words general and more structural properties to require for $\A$ which
yield that the pointwise topology and the Zariski topology on
$\End(\A)$ coincide.

Furthermore, for each $\omega$-categorical structure $\A$ explicitly
considered thus far, it was possible to show that the pointwise
topology and the Zariski topology on $\End(\A)$ coincide, leading to
the authors of~\cite{EJMPP-polish} asking the following question which
formed another essential motivation for the present work:
\begin{question}[{\cite[Question~3.1]{EJMPP-polish}}]\label{q:ejmmp}
  Is there an $\omega$-categorical relational structure $\A$ such that
  the topology of pointwise convergence on $\End(\A)$ is strictly
  finer than the Zariski topology?
\end{question}

\subsection{Our work}
\label{sec:our-work}
We establish two new sets of sufficient conditions on a structure $\A$
under which the Zariski topology and the pointwise topology on
$\End(\A)$ coincide -- so, in particular, under which the pointwise
topology is the coarsest Polish (Hausdorff) semigroup topology on
$\End(\A)$. To this end, we give a new application of so-called
\emph{model-complete cores} which have proved to be a helpful tool not
only in the algebraic theory of constraint satisfaction
problems~\cite{wonderland} but also -- of independent purely
mathematical interest -- in the universal algebraic study of
polymorphism clones of $\omega$-categorical
structures~\cite{BKOPP,BartoPinskerDichotomy} as well as in the
Ramsey-theoretic analysis of $\omega$-categorical
structures~\cite{BodirskyRamsey}. We introduce \emph{structures with
  mobile core} -- a weakening of the standard notion of transitive
structures -- and show that for an $\omega$-categorical structure
without algebraicity whose core is mobile such that the model-complete
core of the structure is either finite or has no algebraicity itself,
the Zariski topology and the pointwise topology on its endomorphism
monoid coincide.

These two cases leave a middle ground open -- namely structures whose
model-complete core is infinite but has algebraicity. Thus, this is
where a positive answer to Question~\ref{q:ejmmp} could be found. And
indeed, we give an example of an $\omega$-categorical structure
without algebraicity whose core is mobile for which the pointwise
topology on the endomorphism monoid is strictly finer than the Zariski
topology. Being transitive as well as homogeneous in a finite
relational language, this structure shows that even these additional
standard \emph{well-behavedness} assumptions are insufficient to
guarantee that the two topologies coincide. This indicates that the
structure of the model-complete core really contains the systematic
reason for the two topologies to be equal.

In Section~\ref{sec:preliminaries}, we formally introduce the relevant
notions, in particular the Zariski topology as well as model-complete
cores. In Section~\ref{sec:two-sets-sufficient}, we prove the positive
results about finite cores and cores without algebraicity stated
above. Finally, Section~\ref{sec:counterexample} contains our
counterexample.

\section{Preliminaries}
\label{sec:preliminaries}

\subsection{Structures, homomorphisms, embeddings, automorphisms}
\label{sec:struct-homom-embedd}
For a function $f\colon A\to B$ between arbitrary sets $A,B$ and a
tuple $\bar{a}=(a_{1},\dots,a_{n})$ in $A$, we denote the
tuple\footnote{In contrast to some related works
  (like~\cite{EJMMMP-zariski,EJMPP-polish}), we denote the evaluation
  of the function $f$ at the element $a$ by $f(a)$ and write
  compositions of functions from right to left,
  i.e.~$fg:=f\circ g:=(a\mapsto f(g(a)))$.}
$(f(a_{1}),\dots,f(a_{n}))$ of evaluations by $f(\bar{a})$ for
notational simplicity.  A \emph{(relational) structure}
$\A=\langle A,(R_{i})_{i\in I}\rangle$ is a domain $A$ (in the
following always finite or countably infinite) equipped with
$m_{i}$-ary relations $R_{i}\subseteq A^{m_{i}}$. If
$\B=\langle B,(S_{i})_{i\in I}\rangle$ is another structure such that
$S_{i}$ also has arity $m_{i}$, we call a function $f\colon A\to B$ a
\emph{homomorphism} and write $f\colon\A\to\B$ if $f$ is compatible
with all $R_{i}$ and $S_{i}$, i.e.~if $\bar{a}\in R_{i}$ implies
$f(\bar{a})\in S_{i}$. A homomorphism $f\colon\A\to\A$ is called an
\emph{endomorphism} of $\A$. We denote the set of all endomorphisms of
$\A$ by $\End(\A)$; it forms a monoid with the composition operation
and the neutral element $\id_{A}$. An \emph{embedding} of $\A$ into
$\B$ is an injective homomorphism $f\colon\A\to\B$ which is
additionally compatible with the complements of $R_{i}$ and $S_{i}$,
equivalently if $f(\bar{a})\in S_{i}$ also implies $\bar{a}\in
R_{i}$. The set of all \emph{self-embeddings} of $\A$, i.e.~of all
embeddings of $\A$ into $\A$, is denoted by $\Emb(\A)$; it also forms
a monoid. An \emph{isomorphism} between $\A$ and $\B$ is a surjective
embedding from $\A$ into $\B$. The set of all \emph{automorphisms} of
$\A$, i.e.~of all isomorphisms between $\A$ and itself, is denoted by
$\Aut(\A)$; it forms a group with the composition operation, the
neutral element $\id_{A}$ and the inversion operation. In the special
case that $\A$ is the structure without any relations, the
endomorphism monoid is the full transformation monoid $A^{A}$, the
self-embedding monoid is the set $\Inj(A)$ of all injective maps
$A\to A$, and the automorphism group is the set $\Sym(A)$ of all
permutations on $A$. A weakening of isomorphic structures is given by
the following notion: Two structures $\A$ and $\B$ are called
\emph{homomorphically equivalent} if there exist homomorphisms
$g\colon\A\to\B$ and $h\colon\B\to\A$.

If $C\subseteq A$, then the \emph{induced substructure} $\C$ of $\A$
on $C$ is the structure with domain $C$ where each relation $R_{i}$ is
replaced by $R_{i}\cap C^{m_{i}}$. If $f\colon\A\to\B$ is a
homomorphism, we will in a slight abuse of notation denote the
substructure of $\B$ on the domain $f(A)$ by $f(\A)$.


\subsection{Topologies}
\label{sec:topologies}
If $S$ is a semigroup, we call a topology $\TT$ on $S$ a
\emph{semigroup topology} (and $(S,\TT)$ a \emph{topological
  semigroup}) if the operation $\cdot\colon S\times S\to S$ is a
continuous map with respect to $\TT$ (where $S\times S$ carries the
product topology).

A natural topology on $\End(\A)$ (and also on
$\Emb(\A),\Aut(\A),\Inj(A),\Sym(A)$) is given by the subspace topology
of the product topology on $A^{A}$ where each copy of $A$ is equipped
with the discrete topology, the so-called \emph{pointwise topology}
which we denote by $\TT_{pw}$ (or $\TT_{pw}\vert_{\End(\A)}$ etc.~if
misunderstandings are possible). In the sequel, we will need to
consider the topological closure of $\Aut(\A)$ with respect to the
pointwise topology within $A^{A}$ (or, equivalently, within $\End(\A)$
since the latter is itself closed in $A^{A}$) which we will call the
$\TT_{pw}$-closure of $\Aut(\A)$ for brevity. We remark that for an
$\omega$-categorical structure $\A$, this closure consists precisely
of the so-called \emph{elementary self-embeddings} of $\A$
(see~\cite{Hodges}).
    
The standard topological basis of $\TT_{pw}$ is given by the sets
\begin{displaymath}
  \set{s\in\End(\A)}{s(\bar{a})=\bar{b}},\qquad \bar{a},\bar{b}\text{
    finite tuples in }A.
\end{displaymath}
It is easy to see that $\TT_{pw}$ is a Polish semigroup topology on
$\End(\A)$.

Now we define the \emph{Zariski topology} central to this paper. For
notational simplicity, we will restrict to monoids.
\begin{definition}\label{def:zariski-top}
  Let $S$ be a monoid.
  \begin{enumerate}[label=(\roman*)]
  \item For $k,\ell\in\N$, $\ell<k$, and for
    $p_{0},\dots,p_{k},q_{0},\dots,q_{\ell}\in S$ as well as
    $\varphi(s):=p_{k}sp_{k-1}s\dots sp_{0}$ and
    $\psi(s):=q_{\ell}sq_{\ell-1}s\dots sq_{0}$ (if $\ell=0$, then
    $\psi(s)=q_{0}$ for all $s\in S$), we define
    \begin{displaymath}
      M_{\varphi,\psi}:=\set{s\in S}{\varphi(s)\neq\psi(s)}.
    \end{displaymath}
  \item The \emph{Zariski topology} on $S$, denoted by
    $\TT_{\text{Zariski}}$, is the topology generated by all sets
    $M_{\varphi,\psi}$. Explicitly, the $\TT_{\text{Zariski}}$-basic
    open sets are the finite intersections of sets $M_{\varphi,\psi}$.
  \end{enumerate}
\end{definition}
In general, the Zariski topology need not be a Hausdorff topology or a
semigroup topology, but suitable weakenings do hold. On the one hand,
it always satisfies the first separation axiom~T1: every singleton set
$\{s_{0}\}$ is $\TT_{\text{Zariski}}$-closed (pick $\varphi(s)=s=1s1$,
where $1$ denotes the neutral element of $S$, and $\psi(s)=s_{0}$). On
the other hand, the left and right \emph{translations},
$\lambda_{t}\colon S\to S$, $s\mapsto ts$ and $\rho_{t}\colon S\to S$,
$s\mapsto st$ (where $t\in S$ is fixed) are continuous with respect to
the Zariski topology: To see this, take arbitrary
$\varphi(s):=p_{k}sp_{k-1}s\dots sp_{0}$ and
$\psi(s):=q_{\ell}sq_{\ell-1}s\dots sq_{0}$ as above and note that
$\lambda_{t}^{-1}(M_{\varphi,\psi})=M_{\tilde{\varphi},\tilde{\psi}}$
where
$\tilde{\varphi}(s):=(p_{k}t)s(p_{k-1}t)s\dots s(p_{1}t)s(p_{0})$ and
$\tilde{\psi}(s):=(q_{\ell}t)s(q_{\ell-1}t)s\dots s(q_{1}t)s(q_{0})$;
similarly for $\rho_{t}$.

By a straightforward argument, the Zariski topology is coarser than
any Hausdorff semigroup topology $\TT$ on $S$: One has to show that
$M_{\varphi,\psi}$ is $\TT$-open. If $s\in M_{\varphi,\psi}$, then
$\varphi(s)\neq \psi(s)$, so there exist $U,V\in\TT$ with
$\varphi(s)\in U$, $\psi(s)\in V$ and $U\cap V=\emptyset$ since $\TT$
is Hausdorff. Then $O:=\varphi^{-1}(U)\cap\psi^{-1}(V)$ is a
$\TT$-open set (by continuity of the semigroup operation) such that
$s\in O\subseteq M_{\varphi,\psi}$.

\subsection{Homogeneity, transitivity and algebraicity}
\label{sec:homog-trans-algebr}
Several important properties of a structure $\A$ can be defined from
the canonical group action of $\Aut(\A)$ on $A^{n}$ for $n\geq 1$ by
evaluation which we write as $\Aut(\A)\actson A^{n}$. We will consider
the (pointwise) \emph{stabiliser} of a set $Y\subseteq A$ (usually
finite), that is
$\Stab(Y):=\set{\alpha\in\Aut(\A)}{\alpha(y)=y\text{ for all }y\in
  Y}$. For a tuple $\bar{a}\in A^{n}$, we further define the
\emph{orbit} of $\bar{a}$ under the action,
$\Orb(\bar{a}):=\set{\alpha(\bar{a})}{\alpha\in\Aut(\A)}$, as well as
the \emph{$Y$-relative orbit}
$\Orb(\bar{a};Y):=\set{\alpha(\bar{a})}{\alpha\in\Stab(Y)}$ where
$Y\subseteq A$. By the characterisation theorem due to Engeler,
Ryll-Nardzewski and Svenonius (see~\cite{Hodges}), a countable
structure $\A$ is $\omega$-categorical if and only if for each
$n\geq 1$, the action $\Aut(\A)\actson A^{n}$ has only finitely many
orbits. We say that $\A$ is a \emph{transitive} structure if the
action $\Aut(\A)\actson A$ has a single orbit. The structure $\A$ is
said to have \emph{no algebraicity} if for any finite $Y\subseteq A$
and any $a\in A\setminus Y$, the $Y$-relative orbit $\Orb(a;Y)$ is
infinite.  Finally, we say that $\A$ is a \emph{homogeneous} structure
if any finite partial isomorphism $m\colon\bar{a}\mapsto\bar{b}$ on
$\A$ can be extended to an automorphism $\alpha\in\Aut(\A)$. It is
easy to see that a homogeneous structure in a \emph{finite
  (relational) language}, i.e.  $\A=\langle A,(R_{i})_{i\in I}\rangle$
with $I$ finite, is automatically $\omega$-categorical.

In the sequel, an important property of $\omega$-categorical
structures without algebraicity will be the existence of ``almost
identical'' embeddings/endomorphisms which can be obtained using a
standard compactness argument.
\begin{lemma}[{\cite[Lemma~3.6]{EJMPP-polish}}]\label{lem:emb-almost-ident}
  Let $\A$ be an $\omega$-categorical structure without
  algebraicity. Then for every $a\in A$, there are $f,g$ in the
  $\TT_{pw}$-closure of $\Aut(\A)$ such that
  $f\vert_{A\setminus\{a\}}=g\vert_{A\setminus\{a\}}$ and
  $f(a)\neq g(a)$.
\end{lemma}
If $f$ and $g$ are as in the previous lemma, then for any
$s\in\End(\A)$ we note that $a\in\Img(s)$ if and only if $sf\neq
sg$. This yields the following fact which will be at the heart of both
proofs in Section~\ref{sec:two-sets-sufficient}.
\begin{lemma}[{contained in~\cite[Proof of
    Lemma~5.3]{EJMMMP-zariski}}]\label{lem:image-zariski-open}
  Let $\A$ be an $\omega$-categorical structure without
  algebraicity. Then for every $a\in A$, the set
  $\set{s\in\End(\A)}{a\in\Img(s)}$ is open in the Zariski topology on
  $\End(\A)$.
\end{lemma}

\subsection{Cores}
\label{sec:cores}
A structure $\C$ is called a \emph{model-complete core} if\footnote{In
  the case that $\C$ is $\omega$-categorical, this means that any
  endomorphism of $\C$ is an elementary self-embedding.} the
endomorphism monoid $\End(\C)$ coincides with the $\TT_{pw}$-closure
of the automorphism group $\Aut(\C)$. If $\C$ is finite, this means
$\End(\C)=\Aut(\C)$. Every $\omega$-categorical structure has a
homomorphically equivalent model-complete core structure:
\begin{theorem}[originally {\cite[Theorem~16]{Cores-journal}},
  alternative proof in
  {\cite[Theorem~5.7]{BKOPP-equations}}]\label{thm:existence-cores}
  Let $\A$ be an $\omega$-categorical structure. Then there exists a
  model-complete core $\C$ such that $\A$ and $\C$ are homomorphically
  equivalent. Moreover, $\C$ is either $\omega$-categorical or finite
  and uniquely determined (up to isomorphism).
\end{theorem}
Because of the uniqueness result, $\C$ is commonly referred to as
\emph{the} model-complete core of $\A$.  We will repeatedly use the
following simple property of model-complete cores:
\begin{lemma}\label{lem:core-hom-emb}
  Let $\A$ be an $\omega$-categorical structure and let $\C$ be its
  model-complete core. Then any homomorphism $f\colon\C\to\A$ is an
  embedding.
\end{lemma}
\begin{proof}
  If $g\colon\A\to\C$ denotes the homomorphism existing by homomorphic
  equivalence, then $gf$ is an endomorphism of $\C$ and thus contained
  in the $\TT_{pw}$-closure of $\Aut(\C)$, in particular a
  self-embedding. This is only possible if $f$ is an embedding.
\end{proof}
This lemma in particular applies to the homomorphism $h\colon\C\to\A$
yielded by homomorphic equivalence. Replacing $\C$ by its isomorphic
copy $h(\C)$, we will subsequently assume that $\C$ is a substructure
of $\A$. Note that depending on the structure $\A$, it can but need
not be possible to pick the homomorphism $g\colon\A\to\C$ to be
surjective. For instance, the model-complete core of the random graph
is the complete graph on countably many vertices, and any surjection
from the random graph to the complete graph is a surjective
homomorphism.
    
On the other hand, if $\A$ is given by the rational numbers $\Q$
extended by two elements $\pm\infty$, equipped with the canonical
strict order, then the model-complete core of $\A$ is precisely
$\langle\Q,<\rangle$ which cannot coincide with any homomorphic image
of $\A$ since such an image would have a greatest and a least
element. If the model-complete core of $\A$ is finite, however,
\emph{any} homomorphism $g\colon\A\to\C$ is surjective, as can be seen
by viewing $g$ as an endomorphism of $\A$ and applying the following
lemma we will also use later on:
\begin{lemma}\label{lem:finite-core-image-minimal}
  If the model-complete core of an $\omega$-categorical structure $\A$
  is finite of size $n$, then the image of any endomorphism of $\A$
  has size at least $n$.
\end{lemma}
\begin{proof}
  If $s\in\End(\A)$, then $s(\A)$ is homomorphically equivalent to
  $\A$. Hence, $s(\A)$ and $\A$ have the same model-complete core
  which can therefore be regarded as a substructure of $s(\A)$.
\end{proof}

\section{Two sets of sufficient conditions}
\label{sec:two-sets-sufficient}
This section is devoted to stating and showing our sufficient
conditions, expressed in terms of the model-complete core, for the
pointwise topology and the Zariski topology to coincide, see
Theorem~\ref{thm:suff-cond-zariski-equ-pw}.

\subsection{Our results}
\label{sec:struct-with-mobile}
An essential notion for our results is given by \emph{structures with
  mobile core}:
\begin{definition}
  Let $\A$ be an $\omega$-categorical structure. Then $\A$ is said to
  have a \emph{mobile core} if any element of $\A$ is contained in the
  image of an endomorphism into the model-complete core. Explicitly,
  for any $a\in\A$, there ought to exist a substructure $\C$ of $\A$
  and $g\in\End(\A)$ with the following properties:
  \begin{enumerate}[label=(\roman*)]
  \item $\C$ is a model-complete core homomorphically equivalent to
    $\A$,
  \item $a\in g(A)\subseteq C$.
  \end{enumerate}
\end{definition}
Note that structures with mobile core are a weakening of transitive
structures (as introduced in Subsection~\ref{sec:homog-trans-algebr}):
Let $\A$ be transitive, let $\C$ be its model-complete core with
homomorphism $g\colon\A\to\C$, and let $a_{0}\in A$ be a fixed
element. If $a\in A$ is arbitrary, then transitivity yields
$\alpha\in\Aut(\A)$ such that $\alpha(g(a_{0}))=a$. Hence,
$\widetilde{\C}:=\alpha(\C)$ is an isomorphic copy of $\C$ with
homomorphism $\tilde{g}:=\alpha g\colon\A\to\widetilde{\C}$ such that
$a\in\tilde{g}(A)\subseteq\widetilde{C}$. In fact, it suffices to
assume that $\A$ is \emph{weakly transitive}, i.e.~that for all
$a,b\in A$ there exists $s\in\End(\A)$ with $s(a)=b$ -- replacing
$\alpha$ in the above argument by $s$, we still obtain that $s(\C)$ is
an isomorphic copy of $\C$ by Lemma~\ref{lem:core-hom-emb}.
    
On the other hand, there exist non-transitive structures which have a
mobile core, for instance the disjoint union of two transitive
structures where each part gets named by an additional unary predicate
(to ascertain that the parts are invariant under any automorphism).
Finally, the structure $\langle\Q\cup\{\pm\infty\},<\rangle$ mentioned
after Lemma~\ref{lem:core-hom-emb} does not have a mobile core: The
element $+\infty$ cannot be contained in any copy of the
model-complete core $\langle\Q,<\rangle$.
    
Now we can formally state the main result of this section.
\begin{theorem}\label{thm:suff-cond-zariski-equ-pw}
  Let $\A$ be an $\omega$-categorical structure without algebraicity
  which has a mobile core. Then the Zariski topology on $\End(\A)$
  coincides with the pointwise topology if one of the following two
  conditions holds:
  \begin{enumerate}[label=(\roman*)]
  \item\label{item:suff-cond-zariski-equ-pw-i} EITHER the
    model-complete core of $\A$ is finite,
  \item\label{item:suff-cond-zariski-equ-pw-ii} OR the model-complete
    core of $\A$ is infinite and does not have algebraicity.
  \end{enumerate}
\end{theorem}
The cases~\ref{item:suff-cond-zariski-equ-pw-i}
and~\ref{item:suff-cond-zariski-equ-pw-ii} will be treated separately
in Subsections~\ref{sec:finite-cores} and~\ref{sec:cores-with-algebr},
respectively. Before we get to the proofs, we show how
Theorem~\ref{thm:suff-cond-zariski-equ-pw} can be used to easily
verify that the Zariski topology and the pointwise topology coincide
on the endomorphism monoids of a multitude of example structures. Some
of them have been treated in~\cite{EJMPP-polish}, but our result
applies to many other structures which have not yet been considered,
e.g. the random $n$-clique-free graph with or without loops.
\begin{corollary}\label{cor:application}
  Let $\A$ be one of the following structures:
  \begin{enumerate}[label=(\roman*)]
  \item\label{item:application-i} $\langle\Q,\leq\rangle$
  \item\label{item:application-ii} the random reflexive partial order
  \item\label{item:application-iii} the equivalence relation with
    either finitely or countably many equivalence classes of countable
    size (for the case of a single class, this includes the complete
    graph on countably many vertices with loops)
  \item\label{item:application-iv} the random (di\nobreakdash-)graph
    with loops
  \item\label{item:application-v} the random $n$-clique-free graph
    with loops
  \item\label{item:application-vi} $\langle\Q,<\rangle$
  \item\label{item:application-vii} the random strict partial order
  \item\label{item:application-viii} the random tournament
  \item\label{item:application-ix} the \emph{irreflexive} equivalence
    relation with either finitely or countably many equivalence
    classes of countable size (for the case of a single class, this
    includes the complete graph on countably many vertices without
    loops)
  \item\label{item:application-x} the random (di\nobreakdash-)graph
    without loops
  \item\label{item:application-xi} the random $n$-clique-free graph
    without loops
  \end{enumerate}
  Then the pointwise topology and the Zariski topology on $\End(\A)$
  coincide. In particular, the pointwise topology is the coarsest
  Hausdorff semigroup topology on $\End(\A)$.
\end{corollary}
\begin{proof}
  It is immediate that all structures
  in~\ref{item:application-i}-\ref{item:application-xi} are
  $\omega$-categorical structures without algebraicity which are
  transitive (in particular, they have a mobile
  core). For~\ref{item:application-i}-\ref{item:application-v}, the
  model-complete core of $\A$ is merely a single point with a loop; in
  particular, the model-complete core is
  finite. For~\ref{item:application-vi} and~\ref{item:application-xi},
  the structure $\A$ is already a model-complete core, so the
  model-complete core of $\A$ is just $\A$
  itself. For~\ref{item:application-vii}
  and~\ref{item:application-viii}, the model-complete core of $\A$ is
  the structure $\langle\Q,<\rangle$.  For~\ref{item:application-ix}
  and~\ref{item:application-x}, the model-complete core of $\A$ is the
  complete graph on countably infinitely vertices. Summarising, the
  model-complete core of $\A$ has no algebraicity
  in~\ref{item:application-vi}-\ref{item:application-xi}.

  In any case, Theorem~\ref{thm:suff-cond-zariski-equ-pw} applies and
  yields the desired conclusion.
\end{proof}

\subsection{Finite cores}
\label{sec:finite-cores}
First, we consider the case that $\A$ has a finite model-complete
core.
\begin{proposition}\label{prop:finite-cores-zariski-equ-pw}
  Let $\A$ be an $\omega$-categorical structure without algebraicity
  which has a mobile core. If the model-complete core of $\A$ is
  finite, then the Zariski topology on $\End(\A)$ coincides with the
  pointwise topology.
\end{proposition}
\begin{proof}
  We show that the $\TT_{pw}$-generating sets
  $\set{s\in\End(\A)}{s(a)=b}$, $a,b\in A$, are
  $\TT_{\text{Zariski}}$-open by proving that they are
  $\TT_{\text{Zariski}}$-neighbourhoods of each element.
  
  Let $s_{0}\in\End(\A)$ such that $s_{0}(a)=b$. Since $\A$ has a
  mobile core, there exist a copy $\C$ of the model-complete core of
  $\A$ and $g\in\End(\A)$ such that $a\in g(A)\subseteq C$. By
  Lemma~\ref{lem:finite-core-image-minimal}, we know that $g(A)=C$. We
  set $n=\abs{C}$ and write $g(\A)=\{a_{1},\dots,a_{n}\}$ where
  $a_{1}=a$. Applying Lemma~\ref{lem:image-zariski-open}, we obtain
  that the set
  \begin{displaymath}
    V:=\set{s\in\End(\A)}{s_{0}(a_{1}),\dots,s_{0}(a_{n})\in\Img(s)}=\bigcap_{j=1}^{n}\set{s\in\End(\A)}{s_{0}(a_{i})\in\Img(s)}
  \end{displaymath}
  is open in the Zariski topology. Since the translation
  $\rho_{g}\colon s\mapsto sg$ on $\End(\A)$ is continuous with
  respect to the Zariski topology, the preimage
  \begin{displaymath}
    U:=\rho_{g}^{-1}(V)=\set{s\in\End(\A)}{s_{0}(a_{1}),\dots,s_{0}(a_{n})\in \Img(sg)}
  \end{displaymath}
  is $\TT_{\text{Zariski}}$-open as well. Again by
  Lemma~\ref{lem:finite-core-image-minimal}, the images of the
  endomorphisms $s_{0}g$ and $sg$ (for arbitrary $s\in\End(\A)$) must
  both have $n$ elements. Hence, the images $s_{0}(a_{i})$ are
  pairwise different and, further,
  \begin{displaymath}
    U=\set{s\in\End(\A)}{\Img(sg)=\{s_{0}(a_{1}),\dots,s_{0}(a_{n})\}}.
  \end{displaymath}
  The crucial observation is that $Ug=\set{sg}{s\in U}$ is a finite
  set: Any element $sg$ is determined by the ordered tuple
  $(s(a_{1}),\dots,s(a_{n}))$. Since the unordered set
  $\{s(a_{1}),\dots,s(a_{n})\}$ is fixed for $s\in U$, there are only
  finitely many (at most $n!$, to be precise) possibilities for the
  ordered tuple.

  Consequently, the set $M:=\set{sg}{s\in U, s(a)\neq b}$ is finite as
  well. We define
  \begin{displaymath}
    O:=U\cap\bigcap_{t\in M}\set{s\in\End(\A)}{sg\neq t}\in\TT_{\text{Zariski}}
  \end{displaymath}
  and claim that $O=\set{s\in\End(\A)}{s\in U, s(a)=b}$, subsequently
  giving $s_{0}\in O\subseteq\set{s\in\End(\A)}{s(a)=b}$ as
  desired. If $s\in U$ with $s(a)=b$, then we take $z\in A$ with
  $g(z)=a$ and note $sg(z)=s(a)=b\neq t(z)$ for all $t\in
  M$. Conversely, if $s\in U$ but $s(a)\neq b$, then $t:=sg\in M$, so
  $s\notin O$ -- completing the proof.
\end{proof}
\subsection{Cores without algebraicity}
\label{sec:cores-with-algebr}
Now we consider structures $\A$ whose model-complete cores do not have
algebraicity. In our proof, we will use the following technical
condition from~\cite{EJMMMP-zariski}:
\begin{lemma}[{\cite[Lemma~5.3]{EJMMMP-zariski}}]\label{lem:technical-EJMMMP}
  Let $X$ be an infinite set and let $S$ be a subsemigroup of $X^{X}$
  such that for every $a\in X$ there exist
  $\alpha,\beta,\gamma_{1},\dots,\gamma_{n}\in S$ for some $n\in\N$
  such that the following hold:
  \begin{enumerate}[label=(\roman*)]
  \item\label{item:technical-EJMMMP-i}
    $\alpha\vert_{X\setminus\{a\}}=\beta\vert_{X\setminus\{a\}}$ and
    $\alpha(a)\neq\beta(a)$
  \item\label{item:technical-EJMMMP-ii} $a\in\Img(\gamma_{i})$ for all
    $i\in\{1,\dots,n\}$;
  \item\label{item:technical-EJMMMP-iii} for every $s\in S$ and every
    $x\in X\setminus\{s(a)\}$, there is $i\in\{1,\dots,n\}$ so that
    $\Img(\gamma_{i})\cap s^{-1}(x)=\emptyset$.
  \end{enumerate}
  Then the Zariski topology of $S$ is the pointwise topology.
\end{lemma}
We remark that~\ref{item:technical-EJMMMP-i} corresponds to
Lemma~\ref{lem:emb-almost-ident} and that the proof proceeds by
constructing the generating sets of the pointwise topology from the
sets $\set{s\in S}{a\in\Img(s)}$ exhibited in
Lemma~\ref{lem:image-zariski-open}.

The fact that the model-complete core does not have algebraicity will come into play
via the following observation:
\begin{lemma}\label{lem:no-alg-two-copies}
  Let $\B$ be a countably infinite structure without algebraicity and
  let $b\in\B$. Then there exist $f,h$ in the $\TT_{pw}$-closure of
  $\Aut(\B)$ such that $f(b)=b=h(b)$ and $f(B)\cap h(B)=\{b\}$ (so
  there exist two copies of $\B$ within $\B$ which only have $b$ in
  common).
\end{lemma}
\begin{proof}
  We enumerate $B=\set{b_{n}}{n\in\N}$ where $b_{1}=b$. First, we
  recursively construct automorphisms
  $\alpha_{n},\beta_{n}\in\Aut(\B)$, $n\in\N$, such that
  $\alpha_{n+1}\vert_{\{b_{1},\dots,b_{n}\}}=\alpha_{n}\vert_{\{b_{1},\dots,b_{n}\}}$,
  $\beta_{n+1}\vert_{\{b_{1},\dots,b_{n}\}}=\beta_{n}\vert_{\{b_{1},\dots,b_{n}\}}$
  and
  $\alpha_{n}(\{b_{1},\dots,b_{n}\})\cap\beta_{n}(\{b_{1},\dots,b_{n}\})=\{b\}$
  for all $n\in\N$. We start by setting
  $\alpha_{1}=\beta_{1}:=\id_{B}$. If $\alpha_{n}$ and $\beta_{n}$ are
  already defined, we put $Y:=\alpha_{n}(\{b_{1},\dots,b_{n}\})$ as
  well as $Z:=\beta_{n}(\{b_{1},\dots,b_{n}\})$. Since $\B$ has no
  algebraicity, the relative orbits $\Orb(\alpha_{n}(b_{n+1});Y)$ and
  $\Orb(\beta_{n}(b_{n+1});Z)$ are infinite, so we can find
  $c_{n+1}\in\Orb(\alpha_{n}(b_{n+1});Y)$ which is not contained in
  $Z$ and then find $d_{n+1}\in\Orb(\beta_{n}(b_{n+1});Z)$ which is
  not contained in $Y\cup\{c_{n+1}\}$. Taking $\gamma\in\Stab(Y)$ with
  $\gamma(\alpha_{n}(b_{n+1}))=c_{n+1}$ as well as $\delta\in\Stab(Z)$
  with $\delta(\beta_{n}(b_{n+1}))=d_{n+1}$, and setting
  $\alpha_{n+1}:=\gamma\alpha_{n}$ as well as
  $\beta_{n+1}:=\delta\beta_{n}$ completes the construction. Finally,
  we set $f:=\lim_{n\in\N}\alpha_{n}$ and
  $h:=\lim_{n\in\N}\beta_{n}$; these maps are contained in the
  $\TT_{pw}$-closure of $\Aut(\B)$ and have the desired properties.
\end{proof}

\begin{proposition}\label{prop:core-no-alg-zariski-equals-pw}
  Let $\A$ be an $\omega$-categorical structure without algebraicity
  which has a mobile core. If the model-complete core of $\A$ is
  infinite and does not have algebraicity, then the Zariski topology
  on $\End(\A)$ coincides with the pointwise topology.
\end{proposition}
\begin{proof}
  We check the assumptions of Lemma~\ref{lem:technical-EJMMMP}.

  Since $\A$ is $\omega$-categorical without algebraicity,
  property~\ref{item:technical-EJMMMP-i} follows from
  Lemma~\ref{lem:emb-almost-ident}.

  For properties~\ref{item:technical-EJMMMP-ii}
  and~\ref{item:technical-EJMMMP-iii}, we fix $a\in A$, set $n=2$ and
  construct $\gamma_{1},\gamma_{2}$. Since $\A$ has a mobile core,
  there exist a copy $\C$ of the model-complete core of $\A$ and
  $g\in\End(\A)$ such that $a\in g(A)\subseteq C$. Since $\C$ has no
  algebraicity, there exist $f,h$ in the $\TT_{pw}$-closure of
  $\Aut(\C)$ such that $f(a)=a=h(a)$ and $f(C)\cap h(C)=\{a\}$ by
  Lemma~\ref{lem:no-alg-two-copies}. Using the homomorphism
  $g\colon\A\to\C$, we set $\gamma_{1}:=fg$ and $\gamma_{2}:=hg$,
  considered as endomorphisms of $\A$. Then $a\in\Img(\gamma_{i})$,
  i.e.~\ref{item:technical-EJMMMP-ii} holds. Suppose now that for some
  $s\in\End(\A)$ and $x\in A$, we have
  $\Img(\gamma_{i})\cap s^{-1}\{x\}\neq\emptyset$ for $i=1,2$. In
  order to prove \ref{item:technical-EJMMMP-iii}, the goal is to show
  $x=s(a)$. We rewrite to obtain the existence of $x_{i}\in A$ with
  $sfg(x_{1})=s\gamma_{1}(x_{1})=x=s\gamma_{2}(x_{2})=shg(x_{2})$. As
  a homomorphism from $\C$ to $\A$, the restriction
  $s\vert_{C}:\C\to\A$ is an embedding by
  Lemma~\ref{lem:core-hom-emb}, in particular injective. Hence,
  \begin{displaymath}
    fg(x_{1})=hg(x_{2})\in f(C)\cap h(C)=\{a\},
  \end{displaymath}
  yielding $x=sfg(x_{1})=s(a)$ as desired.
\end{proof}

\section{Counterexample}
\label{sec:counterexample}
In this section, we give an example of an $\omega$-categorical (even
homogeneous in a finite language) and transitive structure without
algebraicity such that the Zariski topology on its endomorphism monoid
does not coincide with the pointwise topology, thus answering
Question~\ref{q:ejmmp}. By our results in
Section~\ref{sec:two-sets-sufficient}, the model-complete core of this
structure must be infinite and have algebraicity. Informally speaking,
we take a complete graph on countably many vertices where each point
has as fine structure a complete \emph{bipartite} graph on countably
many vertices, see Figure~\ref{fig:structure-G} below.

\subsection{Definitions, notation and preliminary properties}
\label{sec:defin-notat-prel}

We start by formally introducing our structure and giving some
notation.
\begin{definition}\label{def:complete-bipartite}
  \hspace{0mm}
  \begin{enumerate}[label=(\roman*)]
  \item Let $\K_{2,\omega}$ denote the \emph{complete bipartite graph
      on countably many vertices} (without loops): We write the domain
    as $K_{2,\omega}:=A_{+1}\cupdot A_{-1}$ where $A_{+1}$ and
    $A_{-1}$ are countably infinite sets referred to as the
    \emph{parts} of $\K_{2,\omega}$, and define the edge relation as
    $E^{\K_{2,\omega}}:=A_{-1}\times A_{+1}\cup A_{+1}\times A_{-1}$,
    i.e. two points are connected if and only if they are contained in
    different parts of $\K_{2,\omega}$.
  \item Let $\G$ denote the following structure over the language of
    two binary relations: We set $G:=\N\times K_{2,\omega}$ (countably
    many copies of $K_{2,\omega}$) and define the relations as
    follows:
    \begin{align*}
      E_{1}^{\G}&:=\set{((i,x),(j,y))\in G^{2}}{i\neq j},\\
      E_{2}^{\G}&:=\set{((i,x),(j,y))\in G^{2}}{i=j\text{ and }(x,y)\in E^{\K_{2,\omega}}}.
    \end{align*}
    This means that the set of copies of $K_{2,\omega}$ forms a
    complete graph with respect to $E_{1}$ and that each copy
    $\{i\}\times K_{2,\omega}$ of $K_{2,\omega}$ is indeed a copy of
    the graph $\K_{2,\omega}$ (with respect to $E_{2}$); see
    Figure~\ref{fig:structure-G}.
  \end{enumerate}
\end{definition}
\begin{figure}[h]\label{fig:structure-G}
  \begin{center}
    \begin{tikzpicture}
      \coordinate (origin) at (0,0);
      \coordinate (offset1) at (5,0);
      \coordinate (offset2) at (2.5,4);
      \draw[dashed,ultra thick] (origin) -- ($(origin)+(offset1)$) --
      ($(origin)+(offset1)+(offset2)$) -- ($(origin)+(offset2)$) --
      (origin);
      \draw[dashed,ultra thick] (origin) --
      ($(origin)+(offset1)+(offset2)$);
      \draw[dashed,ultra thick] ($(origin)+(offset1)$) --
      ($(origin)+(offset2)$);
      \draw[dashed,ultra thick] (origin) -- (-2,1);
      \draw[dashed,ultra thick] (origin) -- (-2,-1);
      \draw[dashed,ultra thick] ($(origin)+(offset1)$) -- (7,-1);
      \draw[dashed,ultra thick] ($(origin)+(offset1)$) -- (3,-1);
      \draw[dashed,ultra thick] ($(origin)+(offset2)$) -- (0.5,5);
      \draw[dashed,ultra thick] ($(origin)+(offset2)$) -- (4.5,5);
      \draw[dashed,ultra thick] ($(origin)+(offset1)+(offset2)$) -- (9.5,3);
      \draw[dashed,ultra thick] ($(origin)+(offset1)+(offset2)$) -- (9.5,5);
      \begin{scope}[rotate=15]
        \coordinate (Aorigin) at (0,0);
        \coordinate (Aoffset1) at (5,0);
        \coordinate (Aoffset2) at (2.5,4);
        \filldraw[color=black,thick,fill=white] (Aorigin) ellipse (1 and
        1.5);
        \draw[dotted,thick] (-1,0) -- (1,0);
        \coordinate (AP1) at (-0.6,-0.9);
        \coordinate (AQ1) at (-0.5,0.9);
        \coordinate (AP2) at (-0.4,-1.2);
        \coordinate (AQ2) at (0,1.3);
        \coordinate (AP3) at (0.5,-0.7);
        \coordinate (AQ3) at (0.3,0.9);
        \filldraw (AP1) circle (1pt);
        \filldraw (AQ1) circle (1pt);
        \filldraw (AP2) circle (1pt);
        \filldraw (AQ2) circle (1pt);
        \filldraw (AP3) circle (1pt);
        \filldraw (AQ3) circle (1pt);
        \draw (AP1) -- (AQ1);
        \draw (AP1) -- (AQ2);
        \draw (AP1) -- (AQ3);
        \draw (AP2) -- (AQ1);
        \draw (AP2) -- (AQ2);
        \draw (AP2) -- (AQ3);
        \draw (AP3) -- (AQ1);
        \draw (AP3) -- (AQ2);
        \draw (AP3) -- (AQ3);
      \end{scope}      
      \begin{scope}[shift=(offset1),rotate=15]
        \coordinate (Borigin) at (0,0);
        \coordinate (Boffset1) at (5,0);
        \coordinate (Boffset2) at (2.5,4);
        \filldraw[color=black,thick,fill=white] (Borigin)
        ellipse (1 and 1.5);
        \draw[dotted,thick] (-1,0) -- (1,0);
        \coordinate (BP1) at (-0.6,-0.9);
        \coordinate (BQ1) at (-0.5,0.9);
        \coordinate (BP2) at (-0.4,-1.2);
        \coordinate (BQ2) at (0,1.3);
        \coordinate (BP3) at (0.5,-0.7);
        \coordinate (BQ3) at (0.3,0.9);
        \filldraw (BP1) circle (1pt);
        \filldraw (BQ1) circle (1pt);
        \filldraw (BP2) circle (1pt);
        \filldraw (BQ2) circle (1pt);
        \filldraw (BP3) circle (1pt);
        \filldraw (BQ3) circle (1pt);
        \draw (BP1) -- (BQ1);
        \draw (BP1) -- (BQ2);
        \draw (BP1) -- (BQ3);
        \draw (BP2) -- (BQ1);
        \draw (BP2) -- (BQ2);
        \draw (BP2) -- (BQ3);
        \draw (BP3) -- (BQ1);
        \draw (BP3) -- (BQ2);
        \draw (BP3) -- (BQ3);
      \end{scope}
      \begin{scope}[shift=(offset2),rotate=15]
        \coordinate (Corigin) at (0,0);
        \coordinate (Coffset1) at (5,0);
        \coordinate (Coffset2) at (2.5,4);
        \filldraw[color=black,thick,fill=white] (Corigin)
        ellipse (1 and 1.5);
        \draw[dotted,thick] (-1,0) -- (1,0);
        \coordinate (CP1) at (-0.6,-0.9);
        \coordinate (CQ1) at (-0.5,0.9);
        \coordinate (CP2) at (-0.4,-1.2);
        \coordinate (CQ2) at (0,1.3);
        \coordinate (CP3) at (0.5,-0.7);
        \coordinate (CQ3) at (0.3,0.9);
        \filldraw (CP1) circle (1pt);
        \filldraw (CQ1) circle (1pt);
        \filldraw (CP2) circle (1pt);
        \filldraw (CQ2) circle (1pt);
        \filldraw (CP3) circle (1pt);
        \filldraw (CQ3) circle (1pt);
        \draw (CP1) -- (CQ1);
        \draw (CP1) -- (CQ2);
        \draw (CP1) -- (CQ3);
        \draw (CP2) -- (CQ1);
        \draw (CP2) -- (CQ2);
        \draw (CP2) -- (CQ3);
        \draw (CP3) -- (CQ1);
        \draw (CP3) -- (CQ2);
        \draw (CP3) -- (CQ3);
      \end{scope}
      \begin{scope}[shift=($(offset1)+(offset2)$),rotate=15]
        \coordinate (Dorigin) at (0,0);
        \coordinate (Doffset1) at (5,0);
        \coordinate (Doffset2) at (2.5,4);
        \filldraw[color=black,thick,fill=white] (Dorigin)
        ellipse (1 and 1.5);
        \draw[dotted,thick] (-1,0) -- (1,0);
        \coordinate (DP1) at (-0.6,-0.9);
        \coordinate (DQ1) at (-0.5,0.9);
        \coordinate (DP2) at (-0.4,-1.2);
        \coordinate (DQ2) at (0,1.3);
        \coordinate (DP3) at (0.5,-0.7);
        \coordinate (DQ3) at (0.3,0.9);
        \filldraw (DP1) circle (1pt);
        \filldraw (DQ1) circle (1pt);
        \filldraw (DP2) circle (1pt);
        \filldraw (DQ2) circle (1pt);
        \filldraw (DP3) circle (1pt);
        \filldraw (DQ3) circle (1pt);
        \draw (DP1) -- (DQ1);
        \draw (DP1) -- (DQ2);
        \draw (DP1) -- (DQ3);
        \draw (DP2) -- (DQ1);
        \draw (DP2) -- (DQ2);
        \draw (DP2) -- (DQ3);
        \draw (DP3) -- (DQ1);
        \draw (DP3) -- (DQ2);
        \draw (DP3) -- (DQ3);
      \end{scope}
    \end{tikzpicture}
  \end{center}
  \caption{The structure $\G$: complete graph on countably many
    vertices (dashed) where each point has a complete bipartite graph
    on countably many vertices as fine structure (solid).}
\end{figure}
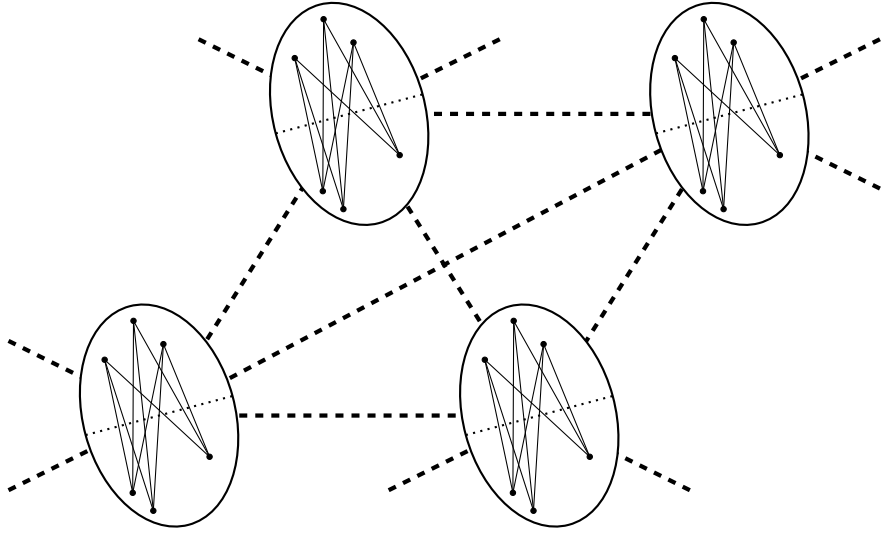
Note that an endomorphism $s$ of $\K_{2,\omega}$ acts as a permutation
on the set $\{A_{+1},A_{-1}\}$ of parts since two
($E_{2}^{\K_{2,\omega}}$-connected) elements from different parts of
$\K_{2,\omega}$ cannot be mapped to the same part of $\K_{2,\omega}$
-- we either have $s(A_{+1})\subseteq A_{+1}$ and
$s(A_{-1})\subseteq A_{-1}$ or $s(A_{+1})\subseteq A_{-1}$ and
$s(A_{-1})\subseteq A_{+1}$.
\begin{definition}\label{def:sign-end-bipartite}
  For $s\in\End(\K_{2,\omega})$, we put $\sgn(s)\in\{+1,-1\}$ to be
  the sign of the permutation induced by $s$ on
  $\{A_{+1},A_{-1}\}$. Explicitly, this means that
  $s(A_{e})\subseteq A_{e\cdot\sgn(s)}$ for $e=\pm 1$. As a slight
  abuse of notation, we will refer to $\sgn(s)$ as the \emph{sign} of
  $s$.
\end{definition}
Clearly, we have $\sgn(st)=\sgn(s)\sgn(t)$ for
$s,t\in\End(\K_{2,\omega})$. As a tool, we define two very simple
endomorphisms of $\K_{2,\omega}$.
\begin{notation}
  \hspace{0mm}
  \begin{enumerate}[label=(\roman*)]
  \item In the sequel, $a_{+1}\in A_{+1}$ and $a_{-1}\in A_{-1}$ shall
    denote fixed elements.
  \item We define $c_{+1}\in\End(\K_{2,\omega})$ and
    $c_{-1}\in\End(\K_{2,\omega})$ to be the unique endomorphisms of
    $\K_{2,\omega}$ with image $\{a_{+1},a_{-1}\}$ and sign $+1$ and
    $-1$, respectively. So $c_{+1}$ is constant on $A_{e}$ with value
    $a_{e}$ and $c_{-1}$ is constant on $A_{e}$ with value $a_{-e}$
    for $e=\pm 1$.
  \end{enumerate}
\end{notation}

In order to describe the automorphism group and endomorphism monoid of
$\G$, the following notation will be useful.
\begin{notation}\label{not:wreath-prod}
  Let $X$ be a set, let $\tau\colon\N\to\N$ and let
  $s_{i}\colon X\to X$ for each $i\in\N$. Then
  $\bigsqcup_{i\in\N}^{\tau} s_{i}$ shall denote the self-map of
  $\N\times X$ defined by
  \begin{displaymath}
    \bigsqcup_{i\in\N}\! ^{\tau} s_{i}:
    \begin{cases}
      \N\times X&\to\N\times X\\
      \hfill (i,x)&\mapsto (\tau(i),s_{i}(x))
    \end{cases}
  \end{displaymath}
  For $\tau\colon\N\to\N$ and $s\colon X\to X$, we further set
  $\tau\ltimes s:=\bigsqcup_{i\in\N}^{\tau}s$.
\end{notation}
\begin{lemma}\label{lem:end-G-aut-G}
  \hspace{0mm}
  \begin{enumerate}[label=(\roman*)]
  \item\label{item:end-G-aut-G-i}
    $\End(\G)=\set{\bigsqcup_{i\in\N}^{\tau}
      s_{i}}{\tau\in\Inj(\N),s_{i}\in\End(\K_{2,\omega})}$.
  \item\label{item:end-G-aut-G-ii}
    $\Aut(\G)=\set{\bigsqcup_{i\in\N}^{\sigma}
      \alpha_{i}}{\sigma\in\Sym(\N),\alpha_{i}\in\Aut(\K_{2,\omega})}$.
  \end{enumerate}
\end{lemma}
\begin{proof}
  It is straightforward to see that the maps
  $\bigsqcup_{i\in\N}^{\tau} s_{i}$ in~\ref{item:end-G-aut-G-i} and
  $\bigsqcup_{i\in\N}^{\sigma}\alpha_{i}$ in~\ref{item:end-G-aut-G-ii}
  form endomorphisms and automorphisms,
  respectively. Thus,~\ref{item:end-G-aut-G-ii} follows immediately
  from~\ref{item:end-G-aut-G-i} since
  $\bigsqcup_{i\in\N}^{\tau} s_{i}$ can only be bijective if
  $\tau\in\Sym(\N)$ and $s_{i}\in\Aut(\K_{2,\omega})$.

  To show~\ref{item:end-G-aut-G-i}, we first note that for any
  $s\in\End(\K_{2,\omega})$ and any two elements $(i,x),(i,y)\in G$ in
  the same copy of $K_{2,\omega}$, the images $s(i,x)$ and $s(i,y)$
  are also contained in the same copy of $K_{2,\omega}$: Either $x,y$
  are connected in $\K_{2,\omega}$ in which case $s(i,x)$ and $s(i,y)$
  are $E_{2}^{\G}$-connected and therefore contained in the same copy,
  or $x,y$ are both connected in $\K_{2,\omega}$ to a common element
  $z$ in which case $s(i,x)$ and $s(i,y)$ are both
  $E_{2}^{\G}$-connected to $s(i,z)$ and therefore contained in the
  same copy. Setting $\tau(i)$ to be the index of this copy,
  i.e.~$s(i,x),s(i,y)\in\{\tau(i)\}\times K_{2,\omega}$, we obtain
  that $s$ can be written as $\bigsqcup_{i\in\N}^{\tau} s_{i}$ for
  some functions $s_{i}\colon K_{2,\omega}\to K_{2,\omega}$. By
  compatibility of $s$ with $E_{1}^{\G}$, the map $\tau$ needs to be
  injective. Further, the maps $s_{i}$ are endomorphisms of
  $\K_{2,\omega}$ since $s$ is compatible with $E_{2}^{\G}$.
\end{proof}
The representation in~\ref{item:end-G-aut-G-ii} readily yields the
following properties of $\G$ by means of lifting from $\Sym(\N)$ and
$\Aut(\K_{2,\omega})$:
\begin{lemma}\label{prop:G-nice}
  $\G$ is $\omega$-categorical, homogeneous, transitive and has no
  algebraicity.
\end{lemma}
\begin{proof}
  We start by showing that $\G$ is homogeneous which will also yield
  the $\omega$-categoricity since $\G$ has a finite language. Let
  $\bar{a}=(a_{1},\dots,a_{n})$ and $\bar{b}=(b_{1},\dots,b_{n})$ be
  tuples in $G$ and let $m\colon\bar{a}\mapsto\bar{b}$ be a finite
  partial isomorphism. Writing $a_{k}=(i_{k},x_{k})$ and
  $b_{k}=(j_{k},y_{k})$, we note that $i_{k}$ and $i_{\ell}$ coincide
  if and only if $j_{k}$ and $j_{\ell}$ coincide (for otherwise,
  either $m$ or $m^{-1}$ would not be compatible with
  $E_{1}^{\G}$). Hence, the map $i_{k}\mapsto j_{k}$ is a well-defined
  finite partial bijection and can thus easily be extended to some
  $\sigma\in\Sym(\N)$ (in other words, the structure with domain $\N$
  and without any relations is homogeneous). Further, if
  $i_{k_{1}}=\ldots=i_{k_{N}}=:i$, then
  $m_{i}\colon x_{k_{1}}\mapsto y_{k_{1}},\dots,x_{k_{N}}\mapsto
  y_{k_{N}}$ is a finite partial isomorphism of $\K_{2,\omega}$ since
  $m$ is a finite partial isomorphism with respect to
  $E_{2}^{\G}$. The graph $\K_{2,\omega}$ is homogeneous, so $m_{i}$
  extends to $\alpha_{i}\in\Aut(\K_{2,\omega})$. Setting
  $\alpha_{i}=\id_{K_{2,\omega}}$ for all $i$ such that no $x_{k}$ is
  contained in the $i$-th copy of $K_{2,\omega}$ and putting
  $\alpha:=\bigsqcup_{i\in\N}^{\sigma}\alpha_{i}\in\Aut(\G)$, we
  obtain an extension of $m$.

  Next, observe that $\G$ is transitive: given $a,b\in G$, the map
  $a\mapsto b$ is a finite partial isomorphism since neither
  $E_{1}^{\G}$ nor $E_{2}^{\G}$ contain any loops. Thus, homogeneity
  yields $\alpha\in\Aut(\G)$ with $\alpha(a)=b$.

  Finally, $\G$ does not have algebraicity since $\K_{2,\omega}$ does
  not have algebraicity: For a finite set $Y\subseteq G$ and
  $a=(i_{0},x_{0})\in G\setminus Y$, we set
  $Y_{i_{0}}:=\set{y\in K_{2,\omega}}{(i_{0},y)\in Y}\not\ni x_{0}$
  and note that $\Orb_{\G}(a;Y)$ encompasses the infinite set
  $\{i_{0}\}\times\Orb_{\K_{2,\omega}}(x_{0};Y_{i_{0}})$ as witnessed
  by the automorphisms $\bigsqcup_{i\in\N}^{\id}\alpha_{i}$ where
  $\alpha_{i_{0}}\in\Stab_{\K_{2,\omega}}(Y_{i_{0}})$ and
  $\alpha_{i}=\id_{K_{2,\omega}}$ for $i\neq i_{0}$.
\end{proof}
\begin{remark}\label{rem:G-nice}
  An alternative construction of $\G$ is as a first-order reduct of
  the \emph{free superposition} (see~\cite{BodirskyRamsey}, this is a
  type of construction to combine two structures with different
  signatures in a ``free'' way) of $\K_{2,\omega}$ with the
  \emph{irreflexive} equivalence relation with countably many
  equivalence classes of countable size. Since both structures are
  transitive and have no algebraicity, the superposition structure has
  the same properties which are then inherited by $\G$ since a
  first-order reduct can only have additional automorphisms.
\end{remark}
To simplify the presentation, we additionally define a few notational
shorthands concerning endomorphisms of $\G$:
\begin{notation}\label{not:shorthands-tilde}
  \hspace{0mm}
  \begin{enumerate}[label=(\roman*)]
  \item For $p=\bigsqcup_{i\in\N}^{\xi}p_{i}\in\End(\G)$, we define
    $\tilde{p}:=\xi\in\Inj(\N)$.
  \item Given $p_{0},\dots,p_{k}\in\End(\G)$ and
    $\varphi(s):=p_{k}sp_{k-1}s\dots sp_{0}$, $s\in\End(\G)$, we
    define
    $\tilde{\varphi}(\tau):=\tilde{p}_{k}\tau\tilde{p}_{k-1}\tau\dots\tau\tilde{p}_{0}$,
    $\tau\in\Inj(\N)$.
  \end{enumerate}
\end{notation}

\subsection{Proof strategy}
\label{sec:proof-strategy}
The goal of Section~\ref{sec:counterexample} is to prove the
following:
\begin{theorem}\label{thm:counterexample}
  On the endomorphism monoid of the structure $\G$, the pointwise
  topology is strictly finer than the Zariski topology.
\end{theorem}
\begin{remark}\label{rem:counterexample-core}
  Before we go into the details of the proof, let us remark that the
  structure $\G$ needs to have an infinite model-complete core which
  has algebraicity in order to have a chance of satisfying
  Theorem~\ref{thm:counterexample} -- for otherwise,
  Theorem~\ref{thm:suff-cond-zariski-equ-pw} would apply.

  The model-complete core of $\K_{2,\omega}$ is just the graph
  consisting of a single edge, as witnessed for instance by the
  substructure induced on $\{a_{+1},a_{-1}\}$ and the homomorphism
  $c_{+1}\colon\K_{2,\omega}\to\{a_{+1},a_{-1}\}$. We claim that the
  model-complete core of $\G$ is the complete graph on countably many
  vertices where each point has as fine structure a single edge,
  i.e. the substructure $\C$ of $\G$ induced on
  $C:=\N\times\{a_{+1},a_{-1}\}\subseteq G$; see
  Figure~\ref{fig:core-of-G}.

  Similarly to the proof of Lemma~\ref{lem:end-G-aut-G}, one easily
  checks that (here, $c_{\pm 1}$ are considered as self-maps of
  $\{a_{+1},a_{-1}\}$)
  \begin{align*}
    \End(\C)&=\set{\bigsqcup_{i\in\N}\!^{\tau}\gamma_{i}}{\tau\in\Inj(\N),\gamma_{i}\in\{c_{+1},c_{-1}\}},\\
    \Aut(\C)&=\set{\bigsqcup_{i\in\N}\!^{\sigma}\gamma_{i}}{\sigma\in\Sym(\N),\gamma_{i}\in\{c_{+1},c_{-1}\}}.
  \end{align*}
  Thus, any endomorphism is locally interpolated by an automorphism,
  and $\C$ is indeed a model-complete core. Additionally, $\G$ and
  $\C$ are homomorphically equivalent -- an example of a homomorphism
  $\G\to\C$ is given by $\bigsqcup_{i\in\N}^{\id}c_{+1}$ (where
  $c_{+1}$ is considered as a map defined on $K_{2,\omega}$).

  Finally, $\C$ has algebraicity: any automorphism of $\C$ which
  stabilises $Y:=\{(0,a_{+1})\}$ also stabilises $a:=(0,a_{-1})$, so
  the $Y$-related orbit of $a$ is finite.
\end{remark}
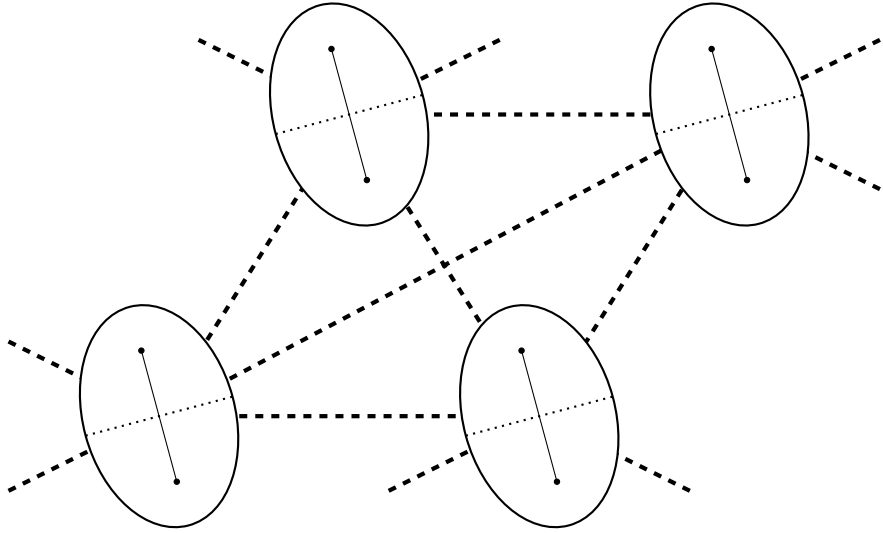
\begin{figure}[h]\label{fig:core-of-G}
  \begin{center}
    \begin{tikzpicture}
      \coordinate (origin) at (0,0);
      \coordinate (offset1) at (5,0);
      \coordinate (offset2) at (2.5,4);
      \draw[dashed,ultra thick] (origin) -- ($(origin)+(offset1)$) --
      ($(origin)+(offset1)+(offset2)$) -- ($(origin)+(offset2)$) --
      (origin);
      \draw[dashed,ultra thick] (origin) --
      ($(origin)+(offset1)+(offset2)$);
      \draw[dashed,ultra thick] ($(origin)+(offset1)$) --
      ($(origin)+(offset2)$);
      \draw[dashed,ultra thick] (origin) -- (-2,1);
      \draw[dashed,ultra thick] (origin) -- (-2,-1);
      \draw[dashed,ultra thick] ($(origin)+(offset1)$) -- (7,-1);
      \draw[dashed,ultra thick] ($(origin)+(offset1)$) -- (3,-1);
      \draw[dashed,ultra thick] ($(origin)+(offset2)$) -- (0.5,5);
      \draw[dashed,ultra thick] ($(origin)+(offset2)$) -- (4.5,5);
      \draw[dashed,ultra thick] ($(origin)+(offset1)+(offset2)$) -- (9.5,3);
      \draw[dashed,ultra thick] ($(origin)+(offset1)+(offset2)$) -- (9.5,5);
      \begin{scope}[rotate=15]
        \coordinate (Aorigin) at (0,0);
        \coordinate (Aoffset1) at (5,0);
        \coordinate (Aoffset2) at (2.5,4);
        \filldraw[color=black,thick,fill=white] (Aorigin) ellipse (1 and
        1.5);
        \draw[dotted,thick] (-1,0) -- (1,0);
        \coordinate (AP) at (0,-0.9);
        \coordinate (AQ) at (0,0.9);
        \filldraw (AP) circle (1pt);
        \filldraw (AQ) circle (1pt);
        \draw (AP) -- (AQ);
      \end{scope}      
      \begin{scope}[shift=(offset1),rotate=15]
        \coordinate (Borigin) at (0,0);
        \coordinate (Boffset1) at (5,0);
        \coordinate (Boffset2) at (2.5,4);
        \filldraw[color=black,thick,fill=white] (Borigin)
        ellipse (1 and 1.5);
        \draw[dotted,thick] (-1,0) -- (1,0);
        \coordinate (BP) at (0,-0.9);
        \coordinate (BQ) at (0,0.9);
        \filldraw (BP) circle (1pt);
        \filldraw (BQ) circle (1pt);
        \draw (BP) -- (BQ);
      \end{scope}
      \begin{scope}[shift=(offset2),rotate=15]
        \coordinate (Corigin) at (0,0);
        \coordinate (Coffset1) at (5,0);
        \coordinate (Coffset2) at (2.5,4);
        \filldraw[color=black,thick,fill=white] (Corigin)
        ellipse (1 and 1.5);
        \draw[dotted,thick] (-1,0) -- (1,0);
        \coordinate (CP) at (0,-0.9);
        \coordinate (CQ) at (0,0.9);
        \filldraw (CP) circle (1pt);
        \filldraw (CQ) circle (1pt);
        \draw (CP) -- (CQ);
      \end{scope}
      \begin{scope}[shift=($(offset1)+(offset2)$),rotate=15]
        \coordinate (Dorigin) at (0,0);
        \coordinate (Doffset1) at (5,0);
        \coordinate (Doffset2) at (2.5,4);
        \filldraw[color=black,thick,fill=white] (Dorigin)
        ellipse (1 and 1.5);
        \draw[dotted,thick] (-1,0) -- (1,0);
        \coordinate (DP) at (0,-0.9);
        \coordinate (DQ) at (0,0.9);
        \filldraw (DP) circle (1pt);
        \filldraw (DQ) circle (1pt);
        \draw (DP) -- (DQ);
      \end{scope}
    \end{tikzpicture}
  \end{center}
  \caption{The model-complete core of $\G$: complete graph on
    countably many vertices (dashed) where each point has a single
    edge (solid) as fine structure.}
\end{figure}
In order to show Theorem~\ref{thm:counterexample}, we will prove that
$\TT_{\text{Zariski}}$-open sets on $\End(\G)$ cannot determine the
sign of the components $s_{i}$ of $s=\bigsqcup_{i\in\N}^{\tau}s_{i}$,
in other words decide whether the functions $s_{i}$ switch the two
parts of $\K_{2,\omega}$ or not. On the other hand,
$\TT_{pw}\vert_{\End(\G)}$-open sets can determine the sign of
finitely many components, thus showing
$\TT_{pw}\vert_{\End(\G)}\neq\TT_{\text{Zariski}}$. More precisely, we
will prove that if a $\TT_{\text{Zariski}}$-generating set
$M_{\varphi,\psi}$ contains $\id_{\N}\ltimes c_{+1}$, then it also
contains $\tau\ltimes c_{-1}$ for all elements $\tau$ of a ``big''
subset of $\Inj(\N)$ -- where ``big'' means either
``$\TT_{pw}\vert_{\Inj(\N)}$-open neighbourhood of $\id_{\N}$'' (if
the terms $\varphi$ and $\psi$ have equal lengths; see
Lemmas~\ref{lem:zariski-sets-equal-length-different-open-nbh}
and~\ref{lem:zariski-sets-equal-length-equal-open-nbh}) or
``$\TT_{pw}\vert_{\Inj(\N)}$-dense and open set'' (if the terms
$\varphi$ and $\psi$ have different lengths; see
Lemma~\ref{lem:zariski-sets-diff-length-dense}).

Our (almost trivial) first lemma analogously holds in a more general
setting. Since we only apply it in case of terms of equal lengths, we
formulate it in the present form.
\begin{lemma}\label{lem:zariski-sets-equal-length-different-open-nbh}
  \hspace{0mm}
  \begin{enumerate}[label=(\roman*)]
  \item Let $k\geq 1$ and let
    $\xi_{0},\dots,\xi_{k},\theta_{0},\dots,\theta_{k}\in\Inj(\N)$ as
    well as
    $\tilde{\varphi}(\tau):=\xi_{k}\tau\xi_{k-1}\tau\dots\tau\xi_{0}$
    and
    $\tilde{\psi}(\tau):=\theta_{k}\tau\theta_{k-1}\tau\dots\tau\theta_{0}$,
    $\tau\in\Inj(\N)$.

    If $\tilde{\varphi}(\id_{\N})\neq\tilde{\psi}(\id_{\N})$, then
    $M_{\tilde{\varphi},\tilde{\psi}}=\set{\tau\in\Inj(\N)}{\tilde{\varphi}(\tau)\neq\tilde{\psi}(\tau)}$
    is a $\TT_{pw}\vert_{\Inj(\N)}$-open neighbourhood of $\id_{\N}$.
  \item Let $k\geq 1$ and let
    $p_{0},\dots,p_{k},q_{0},\dots,q_{k}\in\End(\G)$ as well as
    $\varphi(s):=p_{k}sp_{k-1}s\dots sp_{0}$ and
    $\psi(s):=q_{k}sq_{k-1}s\dots sq_{0}$, $s\in\End(\G)$. Assume
    $\tilde{\varphi}(\id_{\N})\neq\tilde{\psi}(\id_{\N})$ (using the
    shorthand from Notation~\ref{not:shorthands-tilde}).
    
    Then there exists a $\TT_{pw}\vert_{\Inj(\N)}$-open neighbourhood
    $U$ of $\id_{\N}$ such that
    $\tau\ltimes t\in
    M_{\varphi,\psi}=\set{s\in\End(\G)}{\varphi(s)\neq\psi(s)}$ for
    all $\tau\in U$ and $t\in\End(\K_{2,\omega})$. In particular,
    $\tau\ltimes c_{-1}\in M_{\varphi,\psi}$ for all $\tau\in U$.
  \end{enumerate}
\end{lemma}
The second lemma really requires the terms to be of equal length.
\begin{lemma}\label{lem:zariski-sets-equal-length-equal-open-nbh}
  Let $k\geq 1$ and let
  $p_{0},\dots,p_{k},q_{0},\dots,q_{k}\in\End(\G)$ as well as
  $\varphi(s):=p_{k}sp_{k-1}s\dots sp_{0}$ and
  $\psi(s):=q_{k}sq_{k-1}s\dots sq_{0}$, $s\in\End(\G)$. Assume
  $\varphi(\id_{\N}\ltimes c_{+1})\neq\psi(\id_{\N}\ltimes c_{+1})$
  but $\tilde{\varphi}(\id_{\N})=\tilde{\psi}(\id_{\N})$.
    
  Then there exists a $\TT_{pw}\vert_{\Inj(\N)}$-open neighbourhood
  $U$ of $\id_{\N}$ such that
  $\tau\ltimes c_{-1}\in
  M_{\varphi,\psi}=\set{s\in\End(\G)}{\varphi(s)\neq\psi(s)}$ for all
  $\tau\in U$.
\end{lemma}
Finally, we formulate a result for terms of different lengths.
\begin{lemma}\label{lem:zariski-sets-diff-length-dense}
  \hspace{0mm}
  \begin{enumerate}[label=(\roman*)]
  \item Let $\ell<k$ and let
    $\xi_{0},\dots,\xi_{k},\theta_{0},\dots,\theta_{\ell}\in\Inj(\N)$
    as well as
    $\tilde{\varphi}(\tau):=\xi_{k}\tau\xi_{k-1}\tau\dots\tau\xi_{0}$
    and
    $\tilde{\psi}(\tau):=\theta_{\ell}\tau\theta_{\ell-1}\tau\dots\tau\theta_{0}$,
    $\tau\in\Inj(\N)$.

    Then
    $M_{\tilde{\varphi},\tilde{\psi}}=\set{\tau\in\Inj(\N)}{\tilde{\varphi}(\tau)\neq\tilde{\psi}(\tau)}$
    is $\TT_{pw}\vert_{\Inj(\N)}$-dense and open.
  \item Let $\ell<k$ and let
    $p_{0},\dots,p_{k},q_{0},\dots,q_{\ell}\in\End(\G)$ as well as
    $\varphi(s):=p_{k}sp_{k-1}s\dots sp_{0}$ and
    $\psi(s):=q_{\ell}sq_{\ell-1}s\dots sq_{0}$, $s\in\End(\G)$.

    Then there exists a $\TT_{pw}\vert_{\Inj(\N)}$-dense and open set
    $V$ such that
    $\tau\ltimes t\in
    M_{\varphi,\psi}=\set{s\in\End(\G)}{\varphi(s)\neq\psi(s)}$ for
    all $\tau\in V$ and $t\in\End(\K_{2,\omega})$. In particular,
    $\tau\ltimes c_{-1}\in M_{\varphi,\psi}$ for all $\tau\in V$.
  \end{enumerate}
\end{lemma}
We first demonstrate how these auxiliary statements are used and prove
Theorem~\ref{thm:counterexample} before showing the statements
themselves in Subsection~\ref{sec:proof-details}.
\begin{proof}[Proof (of Theorem~\ref{thm:counterexample} given
  Lemmas~\ref{lem:zariski-sets-equal-length-different-open-nbh},~\ref{lem:zariski-sets-equal-length-equal-open-nbh}
  and~\ref{lem:zariski-sets-diff-length-dense})]
  We will show that any $\TT_{\text{Zariski}}$-open set $O$ containing
  $\id_{\N}\ltimes c_{+1}$ also contains $\tau\ltimes c_{-1}$ for some
  $\tau\in\Inj(\N)$. This implies in particular that the
  $\TT_{pw}$-open set $\set{s\in\End(\G)}{s(0,a_{+1})=(0,a_{+1})}$
  cannot be $\TT_{\text{Zariski}}$-open -- proving
  $\TT_{\text{Zariski}}\neq\TT_{pw}$.

  It suffices to consider $\TT_{\text{Zariski}}$-\emph{basic} open
  sets $O$, i.e.
  $O=\bigcap_{h\in H}M_{\varphi_{h},\psi_{h}}\ni\id_{\N}\ltimes
  c_{+1}$ for some finite set $H$. If the terms $\varphi_{h}$ and
  $\psi_{h}$ have equal length, we apply
  Lemma~\ref{lem:zariski-sets-equal-length-different-open-nbh} or
  \ref{lem:zariski-sets-equal-length-equal-open-nbh} to find a
  $\TT_{pw}\vert_{\Inj(\N)}$-open neighbourhood $U_{h}$ of $\id_{\N}$
  such that $\tau\ltimes c_{-1}\in M_{\varphi_{h},\psi_{h}}$ for all
  $\tau\in U_{h}$. If $\varphi_{h}$ and $\psi_{h}$ have different
  lengths, we instead apply\footnote{If $\psi_{h}$ is longer than
    $\varphi_{h}$, we exchange these two terms.}
  Lemma~\ref{lem:zariski-sets-diff-length-dense} to find a
  $\TT_{pw}\vert_{\Inj(\N)}$-dense and open set $V_{h}$ such that
  $\tau\ltimes c_{-1}\in M_{\varphi_{h},\psi_{h}}$ for all
  $\tau\in V_{h}$. Intersecting the respective sets $U_{h}$ and
  $V_{h}$ thus obtained yields a $\TT_{pw}$-open neighbourhood $U$ of
  $\id_{\N}$ and a $\TT_{pw}\vert_{\Inj(\N)}$-dense and open set $V$
  such that $\tau\ltimes c_{-1}\in M_{\varphi_{h},\psi_{h}}$ for all
  $\tau\in U$ whenever $\varphi_{h}$ and $\psi_{h}$ have equal length
  and such that $\tau\ltimes c_{-1}\in M_{\varphi_{h},\psi_{h}}$ for
  all $\tau\in V$ whenever $\varphi_{h}$ and $\psi_{h}$ have different
  lengths. The intersection $U\cap V$ is nonempty; for any
  $\tau\in U\cap V$ we have
  $\tau\ltimes c_{-1}\in M_{\varphi_{h},\psi_{h}}$ for all $h\in H$,
  i.e. $\tau\ltimes c_{-1}\in O$. This concludes the proof.
\end{proof}
\begin{remark}
  A slight refinement of this proof even shows that the Zariski
  topology on $\End(\G)$ is not Hausdorff since
  $\id_{\N}\ltimes c_{+1}$ and $\id_{\N}\ltimes c_{-1}$ cannot be
  separated by open sets: By the proof, a given basic open set around
  $\id_{\N}\ltimes c_{+1}$ contains $\tau\ltimes c_{-1}$ provided that
  $\tau$ is an element of the intersection of a certain
  $\TT_{pw}\vert_{\Inj(\N)}$-open neighbourhood of $\id_{\N}$ and a
  $\TT_{pw}\vert_{\Inj(\N)}$-dense open set. The same idea similarly
  (but with an easier proof in the analogue of
  Lemma~\ref{lem:zariski-sets-equal-length-equal-open-nbh}) yields
  that a given basic open set around $\id_{\N}\ltimes c_{-1}$ contains
  $\tau'\ltimes c_{-1}$ provided that $\tau'$ is an element of the
  intersection of another $\TT_{pw}\vert_{\Inj(\N)}$-open
  neighbourhood of $\id_{\N}$ and another
  $\TT_{pw}\vert_{\Inj(\N)}$-dense open set. The intersection of these
  four sets is nonempty, so the basic open sets around
  $\id_{\N}\ltimes c_{+1}$ and $\id_{\N}\ltimes c_{-1}$ contain a
  common element (namely $\tau\ltimes c_{-1}$ for a certain
  $\tau\in\Inj(\N)$).
\end{remark}
The preceding remark suggests the following refinement of
Question~\ref{q:ejmmp}:
\begin{question}
  Is there an $\omega$-categorical (transitive?) relational structure
  $\A$ such that there exists a Hausdorff (even Polish?) semigroup
  topology on $\End(\A)$ which is \emph{not} finer than the topology
  of pointwise convergence?
\end{question}

\subsection{Proof details}
\label{sec:proof-details}
In this subsection, we prove
Lemmas~\ref{lem:zariski-sets-equal-length-different-open-nbh},~\ref{lem:zariski-sets-equal-length-equal-open-nbh}
and~\ref{lem:zariski-sets-diff-length-dense} in sequence.
\begin{proof}[Proof (of
  Lemma~\ref{lem:zariski-sets-equal-length-different-open-nbh})]
  $ $\par\nobreak\ignorespaces\textbf{(i).} The set
  $M_{\tilde{\varphi},\tilde{\psi}}\subseteq\Inj(\N)$ is open with
  respect to $\TT_{pw}\vert_{\Inj(\N)}$ since $\tilde{\varphi}$ and
  $\tilde{\psi}$ are continuous with respect to
  $\TT_{pw}\vert_{\Inj(\N)}$.

  \textbf{(ii).} Set $U:=M_{\tilde{\varphi},\tilde{\psi}}$ and note
  that if $u:=\varphi(\tau\ltimes t)$ and $v:=\psi(\tau\ltimes t)$,
  then
  $\tilde{u}=\tilde{\varphi}(\tau)\neq\tilde{\psi}(\tau)=\tilde{v}$,
  so $u\neq v$.
\end{proof}
The second lemma requires more work.
\begin{proof}[Proof (of
  Lemma~\ref{lem:zariski-sets-equal-length-equal-open-nbh})]
  We start by fixing some notation. We first write
  $p_{j}=\bigsqcup_{i\in\N}^{\xi_{j}}p_{j,i}$,
  $q_{j}=\bigsqcup_{i\in\N}^{\theta_{j}}q_{j,i}$ (so
  $\xi_{j}=\tilde{p}_{j}$, $\theta_{j}=\tilde{q}_{j}$) and
  $\delta:=\tilde{\varphi}(\id_{\N})=\tilde{\psi}(\id_{\N})$. Further,
  we define $\Xi_{j}:=\xi_{j}\xi_{j-1}\dots\xi_{0}$ as well as
  $\Theta_{j}:=\theta_{j}\theta_{j-1}\dots\theta_{0}$,
  $j=0,\dots,k$. In particular, $\Xi_{k}=\Theta_{k}=\delta$. Let the
  two (distinct, by assumption) functions
  $\varphi(\id_{\N}\ltimes c_{+1})$ and $\psi(\id_{\N}\ltimes c_{+1})$
  differ at the point $(h,x)\in G$. Further, set $e\in\{-1,+1\}$ such
  that $x\in A_{e}$ and choose any $x'\in A_{-e}$.

  In the course of the proof, we will require the explicit expansions
  of the compositions in $\varphi(\id_{\N}\ltimes c_{\pm 1})$ and
  $\psi(\id_{\N}\ltimes c_{\pm 1})$:
  \begin{align*}
    \varphi(\id_{\N}\ltimes c_{\pm
    1})&=\bigsqcup_{i\in\N}\!^{\delta}p_{k,\Xi_{k-1}(i)}c_{\pm
         1}p_{k-1,\Xi_{k-2}(i)}\dots
         c_{\pm 1}p_{0,i}\\
    \psi(\id_{\N}\ltimes c_{\pm
    1})&=\bigsqcup_{i\in\N}\!^{\delta}q_{k,\Theta_{k-1}(i)}c_{\pm
         1}q_{k-1,\Theta_{k-2}(i)}\dots c_{\pm 1}q_{0,i}.
  \end{align*}
 
  We proceed in two steps -- first, we show that
  $\id_{\N}\ltimes c_{-1}\in M_{\varphi,\psi}$; second, we extend this
  to $\tau\ltimes c_{-1}$ for all $\tau$ in an appropriately
  constructed $\TT_{pw}\vert_{\Inj(\N)}$-open neighbourhood of
  $\id_{\N}$.

  \textbf{(1).} $\id_{\N}\ltimes c_{-1}\in M_{\varphi,\psi}$: We
  compare $\varphi(\id_{\N}\ltimes c_{\pm 1})$ and
  $\psi(\id_{\N}\ltimes c_{\pm 1})$ at $(h,x)$ as well as $(h,x')$. In
  order to simplify notation, we define\footnote{$m$ and $n$ count how
    many times the fixed functions (except for the outermost ones)
    involved in evaluating $\varphi(\id_{\N}\ltimes c_{\pm 1})$ and
    $\psi(\id_{\N}\ltimes c_{\pm 1})$ switch the parts of the $h$-th
    copy of $\K_{2,\omega}$.}
  \begin{align*}
    m&:=\sgn(p_{k-1,\Xi_{k-2}(h)})\cdot\sgn(p_{k-2,\Xi_{k-3}(h)})\cdot\ldots\cdot\sgn(p_{0,h})\\
    n&:=\sgn(q_{k-1,\Theta_{k-2}(h)})\cdot\sgn(q_{k-2,\Theta_{k-3}(h)})\cdot\ldots\cdot\sgn(q_{0,h})\\
    \widehat{p}&:=p_{k,\Xi_{k-1}(h)}\\
    \widehat{q}&:=q_{k,\Xi_{k-1}(h)}
  \end{align*}
  and conclude
  \begin{align*}
    [\varphi\left(\id_{\N}\ltimes c_{+1}\right)](h,x) &
                                                        =\Big(\delta(h),\widehat{p}(a_{me})\Big),
    & [\varphi\left(\id_{\N}\ltimes c_{-1}\right)](h,x) &
                                                          =\Big(\delta(h),\widehat{p}(a_{me(-1)^{k-1}})\Big)\\
    [\psi\left(\id_{\N}\ltimes
    c_{+1}\right)](h,x)&=\Big(\delta(h),\widehat{q}(a_{ne})\Big),&\
                                                                   [\psi\left(\id_{\N}\ltimes c_{-1}\right)](h,x)&=\Big(\delta(h),\widehat{q}(a_{ne(-1)^{k-1}})\Big)\\
    [\varphi\left(\id_{\N}\ltimes
    c_{+1}\right)](h,x')&=\Big(\delta(h),\widehat{p}(a_{-me})\Big),&\ [\varphi\left(\id_{\N}\ltimes
                                                                     c_{-1}\right)](h,x')&=\Big(\delta(h),\widehat{p}(a_{-me(-1)^{k-1}})\Big)\\
    [\psi\left(\id_{\N}\ltimes
    c_{+1}\right)](h,x')&=\Big(\delta(h),\widehat{q}(a_{-ne})\Big),&\
                                                                     [\psi\left(\id_{\N}\ltimes c_{-1}\right)](h,x')&=\Big(\delta(h),\widehat{q}(a_{-ne(-1)^{k-1}})\Big)
  \end{align*}
  If
  \begin{displaymath}
    \big\{\widehat{p}(a_{+1}),\widehat{p}(a_{-1})\big\}\neq\big\{\widehat{q}(a_{+1}),\widehat{q}(a_{-1})\big\},
  \end{displaymath}
  then $\varphi(\id_{\N}\ltimes c_{-1})$ and
  $\psi(\id_{\N}\ltimes c_{-1})$ cannot coincide on both $(h,x)$ and
  $(h,x')$, so $\id_{\N}\ltimes c_{-1}\in M_{\varphi,\psi}$ as
  claimed.

  In case of
  \begin{displaymath}
    \big\{\widehat{p}(a_{+1}),\widehat{p}(a_{-1})\big\}=\big\{\widehat{q}(a_{+1}),\widehat{q}(a_{-1})\big\},
  \end{displaymath}
  we distinguish further: If $m=n$, then
  $[\varphi(\id_{\N}\ltimes c_{+1})](h,x)\neq[\psi(\id_{\N}\ltimes
  c_{+1})](h,x)$ shows
  \begin{displaymath}
    \widehat{p}(a_{+1})=\widehat{q}(a_{-1})\quad\text{as
      well as}\quad
    \widehat{p}(a_{-1})=\widehat{q}(a_{+1})
  \end{displaymath}
  which leads to\footnote{Here we use that $\varphi$ and $\psi$ have
    equal lengths (or more precisely: lengths of equal parity).}
  $[\varphi(\id_{\N}\ltimes c_{-1})](h,x)\neq[\psi(\id_{\N}\ltimes
  c_{-1})](h,x)$, so $\id_{\N}\ltimes c_{-1}\in M_{\varphi,\psi}$ as
  claimed. If on the other hand $m=-n$, then we analogously obtain
  \begin{displaymath}
    \widehat{p}(a_{+1})=\widehat{q}(a_{+1})\quad\text{as
      well as}\quad
    \widehat{p}(a_{-1})=\widehat{q}(a_{-1})
  \end{displaymath}
  and
  $[\varphi(\id_{\N}\ltimes c_{-1})](h,x)\neq[\psi(\id_{\N}\ltimes
  c_{-1})](h,x)$, so $\id_{\N}\ltimes c_{-1}\in M_{\varphi,\psi}$ as
  claimed.

  \textbf{(2).}  There exists a $\TT_{pw}\vert_{\Inj(\N)}$-open
  neighbourhood $U\subseteq\Inj(\N)$ of $\id_{\N}$ such that
  $\tau\ltimes c_{-1}\in M_{\varphi,\psi}$ for all $\tau\in U$: One
  immediately checks that for arbitrary $t\in\End(\K_{2,\omega})$, the
  map $\chi_{t}\colon\Inj(\N)\to\End(\G)$,
  $\chi_{t}(\tau):=\tau\ltimes t$ is continuous with respect to
  $\TT_{pw}\vert_{\Inj(\N)}$ and\footnote{Caution! We briefly consider
    the pointwise topology on $\End(\G)$ instead of the Zariski
    topology.} $\TT_{pw}\vert_{\End(\G)}$. Since $M_{\varphi,\psi}$ is
  open with respect to $\TT_{pw}\vert_{\End(\G)}$, the preimage
  $U:=\chi_{c_{-1}}^{-1}(M_{\varphi,\psi})\subseteq\Inj(\N)$ is open
  with respect to $\TT_{pw}\vert_{\Inj(\N)}$. By~(1), the set $U$
  contains $\id_{\N}$ -- completing the proof.
\end{proof}
Finally, we show the third lemma.
\begin{proof}[Proof (of
  Lemma~\ref{lem:zariski-sets-diff-length-dense})]
  $ $\par\nobreak\ignorespaces\textbf{(i).}  We have to prove that for
  two tuples $\bar{z},\bar{w}$ of the same length such that $\bar{w}$
  does not contain the same value twice (since we are working in
  $\Inj(\N)$), the intersection
  $\set{\tau\in\Inj(\N)}{\tau(\bar{z})=\bar{w}}\cap
  M_{\tilde{\varphi},\tilde{\psi}}$ is nonempty. The idea behind the
  proof is to find an element $x_{0}\in\N$ and inductively construct a
  partial injection $\widehat{\tau}$ which extends
  $\bar{z}\mapsto\bar{w}$ such that the values
  \begin{displaymath}
    [\tilde{\varphi}(\widehat{\tau})](x_{0})=\xi_{k}\widehat{\tau}\xi_{k-1}\widehat{\tau}\dots\widehat{\tau}\xi_{0}(x_{0})\quad\text{and}\quad [\tilde{\psi}(\widehat{\tau})](x_{0})=\theta_{\ell}\widehat{\tau}\theta_{\ell-1}\widehat{\tau}\dots\widehat{\tau}\theta_{0}(x_{0})
  \end{displaymath}
  are welldefined (i.e. $\xi_{0}(x_{0})\in\Dom(\widehat{\tau})$,
  $\xi_{1}\widehat{\tau}\xi_{0}(x_{0})\in\Dom(\widehat{\tau})$ etc.)
  and
  $[\tilde{\varphi}(\widehat{\tau})](x_{0})\neq
  [\tilde{\psi}(\widehat{\tau})](x_{0})$. This gives
  $\tau(\bar{z})=\bar{w}$ and
  $\tau\in M_{\tilde{\varphi},\tilde{\psi}}$ for any $\tau\in\Inj(\N)$
  extending $\widehat{\tau}$.

  More precisely, we will define (not necessarily distinct) elements
  $x_{0},\dots,x_{k},x'_{0},\dots,x'_{k}\in\N$ and
  $y_{0}\dots,y_{\ell},y'_{0},\dots,y'_{\ell}\in\N$ such that
  \begin{enumerate}[label=(\arabic*)]
  \item\label{item:proof-zariski-sets-diff-length-dense-i}
    $x_{0}=y_{0}$.
  \item\label{item:proof-zariski-sets-diff-length-dense-ii}
    $x'_{j}=\xi_{j}(x_{j})$ for all $j=0,\dots,k$.
  \item\label{item:proof-zariski-sets-diff-length-dense-iii}
    $y'_{j}=\theta_{j}(y_{j})$ for all $j=0,\dots,\ell$.
  \item\label{item:proof-zariski-sets-diff-length-dense-iv}
    $\widehat{\tau}$ defined by $\bar{z}\mapsto\bar{w}$,
    $(x'_{0},\dots,x'_{k-1})\mapsto (x_{1},\dots,x_{k})$,
    $(y'_{0},\dots,y'_{\ell-1})\mapsto (y_{1},\dots,y_{\ell})$ is a
    welldefined\footnote{This means that if e.g. $x'_{0}=y'_{0}$, then
      $x_{1}=y_{1}$.} partial injection.
  \item\label{item:proof-zariski-sets-diff-length-dense-v}
    $x'_{k}\neq y'_{\ell}$. (This will crucially depend on the
    assumption $\ell<k$.)
  \end{enumerate}
  We first pick $x_{0}=y_{0}\in\N$ such that
  $x'_{0}:=\xi_{0}(x_{0})\notin\bar{z}$ and
  $y'_{0}:=\theta_{0}(y_{0})\notin\bar{z}$; this is possible since the
  set $\xi_{0}^{-1}(\bar{z})\cup\theta_{0}^{-1}(\bar{w})$ of forbidden
  points is finite by injectivity of $\xi_{0}$ and $\theta_{0}$. Note
  that $x'_{0}$ and $y'_{0}$ are not necessarily different (in
  particular, $\xi_{0}=\theta_{0}$ is possible).

  Suppose that $1\leq i\leq\ell$ and that
  $x_{0},\dots,x_{i-1},x'_{0},\dots,x'_{i-1}$ as well as
  $y_{0},\dots,y_{i-1},y'_{0},\dots,y'_{i-1}$ are already defined such
  that
  \ref{item:proof-zariski-sets-diff-length-dense-i}-\ref{item:proof-zariski-sets-diff-length-dense-iv}
  hold (with $i-1$ in place of both $k$ and $\ell$). We abbreviate
  $X_{i-1}:=\{x_{0},\dots,x_{i-1}\},X'_{i-1}:=\{x'_{0},\dots,x'_{i-1}\}$
  and
  $Y_{i-1}:=\{y_{0},\dots,y_{i-1}\},Y'_{i-1}:=\{y'_{0},\dots,y'_{i-1}\}$.
  Pick $x_{i},y_{i}\notin\bar{w}\cup X_{i-1}\cup Y_{i-1}$ such that
  $x'_{i}:=\xi_{i}(x_{i})\notin\bar{z}\cup X'_{i-1}\cup Y'_{i-1}$ and
  $y'_{i}:=\theta_{i}(y_{i})\notin\bar{z}\cup X'_{i-1}\cup Y'_{i-1}$
  with the additional property that\footnote{This ensures that
    $\widehat{\tau}$ is welldefined and injective.} $x_{i}$ and
  $y_{i}$ are chosen to be distinct if and only if $x'_{i-1}$ and
  $y'_{i-1}$ are distinct (to obtain a welldefined partial injection
  in~\ref{item:proof-zariski-sets-diff-length-dense-iv}). As with the
  construction of $x_{0}$ above, this is possible by finiteness of the
  forbidden sets.

  If $\ell+1\leq i\leq k$ and if
  $x_{0},\dots,x_{i-1},x'_{0},\dots,x'_{i-1}$ as well as
  $y_{0},\dots,y_{\ell},y'_{0},\dots,y'_{\ell}$ are already defined
  such that
  \ref{item:proof-zariski-sets-diff-length-dense-i}-\ref{item:proof-zariski-sets-diff-length-dense-iv}
  hold (with $i-1$ in place of $k$), then we again abbreviate
  $X_{i-1}:=\{x_{0},\dots,x_{i-1}\},X'_{i-1}:=\{x'_{0},\dots,x'_{i-1}\}$
  and
  $Y_{\ell}:=\{y_{0},\dots,y_{\ell}\},Y'_{\ell}:=\{y'_{0},\dots,y'_{\ell}\}$. Analogously
  to the previous step, we pick
  $x_{i}\notin\bar{w}\cup X_{i-1}\cup Y_{\ell}$ such that
  $x'_{i}:=\xi_{i}(x_{i})\notin\bar{z}\cup X'_{i-1}\cup
  Y'_{\ell}$. Note that in the final step $i=k$, we are picking
  $x_{k}$ such that\footnote{At this point, it is crucial that
    $\ell<k$ since we would never enter the second phase
    $\ell+1\leq i\leq k$ of the construction otherwise (more
    precisely, if $\ell$ were equal to $k$, we would have to determine
    $x'_{k}$ at the same time as $y'_{\ell}$ and could not make sure
    that they are different).}
  $x'_{k}\notin\bar{z}\cup X'_{k-1}\cup Y'_{\ell}$. In particular, we
  require $x'_{k}\neq y'_{\ell}$,
  i.e.~\ref{item:proof-zariski-sets-diff-length-dense-v}.

  \textbf{(ii).}  The set
  $V:=M_{\tilde{\varphi},\tilde{\psi}}\subseteq\Inj(\N)$ is
  $\TT_{pw}\vert_{\Inj(\N)}$-dense by the first statement and clearly
  $\TT_{pw}\vert_{\Inj(\N)}$-open. For $\tau\in V$, we set
  $u:=\varphi(\tau\ltimes c_{-1})$ as well as
  $v:=\psi(\tau\ltimes c_{-1})$ and note that
  $\tilde{u}=\tilde{\varphi}(\tau)\neq\tilde{\psi}(\tau)=\tilde{v}$. This
  yields $\tau\ltimes c_{-1}\in M_{\varphi,\psi}$ as desired.
  
\end{proof}

\bibliographystyle{alpha}
\bibliography{global.bib}
\end{document}